\documentclass[reqno]{amsart}
\usepackage{amsmath}
\usepackage{hyperref}
\usepackage{amsfonts}
\usepackage{amsthm}
\usepackage{amssymb}
\usepackage{graphicx}
\usepackage{bm}
\usepackage{dsfont}
\usepackage{color}
\usepackage{enumerate}
\usepackage[font=footnotesize]{caption}
\usepackage[all]{xy}\usepackage{scalerel,stackengine}
\usepackage{tikz}
\usepackage{pgfplots}
\usepackage{todonotes}
\allowdisplaybreaks

\numberwithin{equation}{section}
\newtheorem{thm}{Theorem}[section]
\newtheorem*{thm*}{Theorem}

\newtheorem{lem}[thm]{Lemma}
\newtheorem{prop}[thm]{Proposition}
\theoremstyle{definition}
\newtheorem{rem}[thm]{Remark}

\newtheorem{dfn}[thm]{Definition}
\newtheorem{rmk}[thm]{Remark}
\newcommand{\N}{\mathds{N}}

\newcommand{\Z}{\mathds{Z}}

\newcommand{\R}{\mathds{R}}

\newcommand{\pa}{\partial}

\newcommand{\C}{\mathds{C}}
\newcommand{\T}{\mathds{T}}

\newcommand{\diff}{\mathrm{d}}

\newcommand{\be}{\begin{equation}}
\newcommand{\ee}{\end{equation}}

\stackMath
\newcommand\reallywidehat[1]{%
\savestack{\tmpbox}{\stretchto{%
  \scaleto{%
    \scalerel*[\widthof{\ensuremath{#1}}]{\kern-.6pt\bigwedge\kern-.6pt}%
    {\rule[-\textheight/2]{1ex}{\textheight}}%WIDTH-LIMITED BIG WEDGE
  }{\textheight}% 
}{0.5ex}}%
\stackon[1pt]{#1}{\tmpbox}%
}
\begin{document}

\title[Zoll magnetic systems]{Zoll magnetic systems on the two-torus: \\ 
a Nash--Moser construction }

%\title[Zoll magnetic systems]{Zoll magnetic systems on the two-torus: \\ 
%a construction via the Nash--Moser Implicit Function Theorem}

\author[L. Asselle]{Luca Asselle}
\address{Luca Asselle\newline\indent
Ruhr-Universit\"at Bochum, Fakult\"at f\"ur Mathematik, Universit\"atsstra\ss e 150, 44801 Bochum}
\email{luca.asselle@rub.de}

\author[G. Benedetti]{Gabriele Benedetti}
\address{Gabriele Benedetti\newline\indent 
Vrije Universiteit Amsterdam, Department of Mathematics, De Boelelaan 1111, 1081 HV Amsterdam}
\email{g.benedetti@vu.nl}
\author[M. Berti]{Massimiliano Berti}
\address{Massimiliano Berti\newline\indent 
	SISSA, Via Bonomea 265, 34136, Trieste, Italy.}
\email{berti@sissa.it}
%\date{\today}
%\subjclass[2000]{37J45, 58E05}

\begin{abstract}
%In this paper 
We construct an infinite-dimensional family of smooth 
integrable {\it magnetic} systems on the two-torus which 
are {\it Zoll}, meaning that all the unit-speed
magnetic geodesics are periodic. 
The metric and the magnetic field of such systems are arbitrarily close to the flat metric and to a given constant magnetic field.
This extends to the magnetic setting a famous result by Guillemin \cite{Guillemin:1976} on the two-sphere. We characterize Zoll magnetic systems as zeros of a suitable action functional $S$,
and then look for its zeros  by means of a Nash--Moser implicit function theorem. 
This requires  showing the right-invertibility of the linearized operator $\diff S$ in a
neighborhood of the flat metric and constant magnetic field, 
and  establishing tame estimates for the right inverse.  
As  key step  we prove the invertibility  of the normal
operator $\diff S\circ \diff S^*$ which,  
unlike in Guillemin's case, is % {\it not} 
pseudo-differential only at the highest order.  
We  overcome this difficulty
noting that, by % exploiting 
the asymptotic 
properties of Bessel functions,  the lower order expansion of % we split
 $ \diff S \circ \diff S^*$ is a  sum of 
Fourier integral operators.  
We then use
a resolvent identity decomposition which reduces  the problem to the invertibility of
$\diff S \circ \diff S^*$  restricted to the subspace of functions 
 corresponding to high Fourier modes. The inversion of  such a restricted operator
is finally achieved 
by making the crucial observation 
 that lower order  Fourier integral operators 
satisfy {\it asymmetric} tame estimates. 
\end{abstract}

\maketitle

\vspace{-5mm}

\tableofcontents

\vspace{-8mm}

\section{Introduction and main result}

A \textit{magnetic system} on a closed oriented surface $\Sigma$ is a pair $ (g,f) $ where $ g $ is a Riemannian metric on $\Sigma$
and $ f : \Sigma \to \R $ is a smooth function, called the \textit{magnetic function}. A \textit{magnetic geodesic} is a unit-speed curve $\gamma\colon\R\to \Sigma$ 
satisfying the equation
\begin{equation}\label{e:maggeo}
\nabla_t \dot \gamma = -f(\gamma)\, \jmath \dot\gamma \, ,
\end{equation}
where $\nabla_t$ denotes the Levi-Civita covariant derivative of $g$ and $\jmath\colon T\Sigma\to T\Sigma$ is the  counter clock-wise rotation by ninety-degree. From a physical point of view, magnetic geodesics model the motion of a charged particle in $\Sigma$ under the effect of the magnetic field $ -f\mu$, where $\mu$ is the area form on $\Sigma$. From a geometric point of view, magnetic geodesics are curves 
with {\it prescribed} % whose 
geodesic curvature % $\kappa_\gamma$ is prescribed by $-f$, i.e., they solve 
$$
\kappa_\gamma=-f(\gamma) \, . 
$$
The goal of the present paper is to construct an infinite-dimensional space of 
smooth magnetic systems whose dynamics is as simple as possible.
\begin{dfn} {\bf (Zoll magnetic system)}
A magnetic system $(g,f)$ is called \textit{Zoll} if, up to a smooth time reparametrization, all of its magnetic geodesics are periodic with the same 
minimal period. 
\end{dfn}

We notice that, if $\Sigma $ is not the two-sphere $ {\mathbb S}^2$, a system is Zoll if and only if all of its magnetic geodesics are periodic, see \cite{Asselle:2020b}. 
In contrast, on $ {\mathbb S}^2$ this condition is not sufficient as the Katok examples show \cite{Ben}.

Zoll magnetic systems generalize the classical notion of Zoll Riemannian metrics, namely those metrics such that all geodesics are periodic and with the same length. 
For topological reasons, the only oriented surface that can carry Zoll metrics is the two-sphere ${\mathbb S}^2$, see \cite{Besse}. 
The first explicit examples of 
Zoll metrics on $\mathbb{S}^2$ different from the round metric $g_0$ have been found by Zoll in \cite{Zoll:1903} within the class of metrics of revolution. Based on the work of Funk \cite{Funk:1913}, 
who constructed formal power series expansions of Zoll metrics close to  $ g_0 $, the breakthrough in the study of Zoll metrics 
% close to the standard metric $ g_0 $
is due to Guillemin in \cite{Guillemin:1976} 
via an implicit function theorem approach \`a la Nash--Moser: given a smooth function $u\colon {\mathbb S}^2\to\R$, there is a one-parameter
 family of Zoll metrics on ${\mathbb S}^2 $, 
 \[
 g(\tau):=(1+\tau u+o(\tau))g_0 \, , \qquad |\tau| \ll 1 \, , 
 \] 
 with fixed total area for any $\tau$, % small enough
 if and only if $u$ is an odd function. 
Indeed,  the function $ u $ has to belong to the kernel of the Radon transform,  defined as the integral of a function on all the oriented closed geodesics of $ g_0 $, which are the standard circles.
A far-reaching
generalization of Guillemin's result to higher-dimensional spheres 
 $ {\mathbb S}^n $ has recently been 
 obtained by Ambrozio--Marquez--Neves in
 \cite{Ambrozio:2021},  where % it is shown how to construct
 metrics  admitting Zoll families of minimal hypersurfaces are constructed, 
 exploiting  properties of the higher-dimensional Radon transform.
We also mention that  a constructive approach to Zoll metrics using complex geometry
has been developed in \cite{LBM}, and we quote   \cite{MS22}
 for a variational characterization of Zoll metrics.

Zoll metrics play a crucial role in systolic geometry, which estimates the length of the \textit{systole} of a Riemannian manifold (that is, the length of the shortest non-constant closed geodesics) in terms of global geometric invariants, see e.g., the monograph \cite{Katz}. Indeed, Zoll metrics on $\mathbb S^2$ are local maximizers of the systolic ratio, namely the ratio between the square of the systole and the total area, see \cite{ABHS,ABHS21}. 
 Similarly, Zoll magnetic systems are particularly important 
 being local maximizers of the magnetic systolic ratio \cite{BKmag}, in which the length is replaced by the magnetic length and the total area by a suitable global invariant involving the Euler characteristic of the surface, the total area and the integral of the magnetic function. Both results follow from
 recent breakthroughs \cite{APB,ABcon,BKodd} on the study of local systolic inequalities in contact and symplectic geometry. 
 
 In contrast to Zoll metrics, {\it no} examples of Zoll magnetic systems are known so far besides:
 \begin{enumerate}
 \item 
 the trivial ones where the Gaussian curvature $K$ and the magnetic field $f$ are constant functions satisfying $K+f^2>0$, so that the magnetic geodesics are 
geodesic circles (such 
trivial  examples are actually available on any orientable surface); 
\item the few explicit 
 \textit{integrable} examples on a flat two-torus exhibited in \cite{Asselle:2020}. 
 \end{enumerate}
 The goal of the present work is to construct an abundance of smooth 
 integrable Zoll magnetic systems on the two-torus, where the metric is {\it not} flat 
and the magnetic field is {\it not} constant.
We now introduce the 
class of integrable magnetic systems and  we state precisely our main result, which is Theorem \ref{thm:main}.

{\bf Integrable magnetic systems on the two-torus.} Let $\T:=\R/2\pi\Z$ and 
consider magnetic systems  on the infinite cylinder $ \T\times\R$ 
with coordinates $ (x, y) $ 
given by
\begin{equation}\label{magsys1}
g = \diff x^2 + A^2(x) \, \diff y^2, \qquad f=f(x) \, ,
\end{equation}
where $A\colon \T\to (0,+\infty)$ and $f\colon\T\to\R$ are smooth functions. 
If $ (g,f) $ is Zoll then $\int_\T A(x)f(x)\, \diff x  \neq 0$ by \cite{Asselle:2020} 
(in particular the magnetic function $f$ has to be non-zero).
Thus, up to a dilation $(x,y)\mapsto (x,cy)$ which changes $A$ to $c^2A$ and $f$ to $\mathrm{sign}(c)f$ we can assume the normalization 
\begin{equation}
\frac{1}{2\pi}\int_\T A(x)f(x)\, \diff x = 1 \, . 
\label{constraint1}
\end{equation}
We identify  the quotient of $\T\times\R$ by the translation $(x,y)\mapsto (x,y+2\pi)$
with the two-torus. 
Since the metric $g$ and the magnetic function $f$ in \eqref{magsys1} are invariant under all vertical translations 
they yield an ``integrable" magnetic system on the two-torus. It is convenient to describe the magnetic system $(g,f)$ through the pair $(A,B)$, where $A$ is in \eqref{magsys1} and $B\colon\T\to\T$ is any smooth map such that 
\[
B' (x) = A(x)f (x)   \qquad  {\rm where} \qquad  ' := \frac{\diff}{\diff x }	\,  .
\]
By \eqref{constraint1} we see that $B $ is a well-defined map  on $ \T $  and
is determined by $(g,f)$ only up to an additive constant. 

Note that, for any  constant $A_*>0 $ the pair $ (A_*,  {\rm id}_\T ) $, 
where ${\rm id}_\T   \colon\T\to\T $ 
 denotes 
 the identity map on the torus $ \T $, 
 is a trivial Zoll pair, where the metric is flat and the magnetic function is constant and equal to $1/A_*$. In particular, the magnetic geodesics are 
ellipses with semi-axes of length $A_*$ and $1$ and parallel to the $x$- and $y$-direction, respectively.
%The main result of this paper proves the existence of 
%that nearby there is an abundance of Zoll magnetic systems.  
 
{\bf Main result.} 
We introduce for any $A_*>0$ the infinite-dimensional space
\be\label{SDABstar}
V_{A_*} % \ker \diff S(A_*,{\rm id}_\T ) 
: = \Big \{(\alpha,\beta) \in C^\infty(\T)\times C^\infty 
(\T) \ : \ J_1'(k A_*) \widehat{\alpha}(k) - i J_1(k A_*) \widehat{\beta}(k) =0,\ \ \forall k \in\Z\setminus\{0\}\Big \}
\ee
where $ \widehat{\alpha}(k) $,  $\widehat{\beta}(k) $ denote the Fourier coefficients of $ \alpha$ and $\beta$, and
$J_1:\R\to \R$ is the first Bessel function
\be\label{Bess1}
J_1 (\theta) := \frac{1}{2 \pi} \int_{\T} e^{i \theta \sin \varphi} e^{- i \varphi }
\diff \varphi  \, . 
\ee

Our main result proves that the set of integrable Zoll magnetic systems is a subset of 
$ C^\infty  (\T) \times C^\infty (\T) $ which is tangent to  $ V_{A_*} $ at any 
trivial Zoll magnetic system $ (A_*,  {\rm id}_\T ) $.

\begin{thm}\label{thm:main}
{\bf (Integrable Zoll magnetic systems)}
For any $ A_* > 0 $  and  any
$ (\alpha, \beta) \in V_{A_*} $ defined in
\eqref{SDABstar},  there exists a smooth one-parameter family of 
{\it Zoll} magnetic systems on the two-torus of the form 
\begin{equation}\label{abmain}
(A(\tau), B(\tau)) := \big( A_* + a(\tau)  , {\rm id}_\T + b(\tau)  \big) 
\, , \quad \tau \in (- \delta, \delta) \, , \quad (a(\tau), b(\tau)) \in C^\infty (\T) 
\times C^\infty (\T) \, , 
\end{equation}
with 
\[ 
(a(0), b(0)) = (0,0) \qquad \text{and}\qquad \frac{\mathrm d}{\mathrm d \tau} \Big |_{\tau=0} (a(\tau), b(\tau)) = (\alpha, \beta) \, .
\]
\end{thm}

Let us make some comments on Theorem~\ref{thm:main}. 
A formal power series construction of the functions 
$ (a(\tau), b(\tau) ) $ in \eqref{abmain}
seems possible by 
implementing Hamiltonian perturbation theory, 
for example in the spirit of
\cite{Loc92,BamGio}, but the problem of its convergence looks completely open, 
as in the geodesic work of Funk  \cite{Funk:1913}.
For this reason we assume a completely different point of view. 
\begin{enumerate}
\item   {\bf (Action functional $ S $)}
We construct the 
Zoll magnetic systems \eqref{abmain} as zeros of the ``action functional" $ S(a,b) $ defined  in \eqref{eq:S} below, which is 
expressed in terms of the Bessel function 
in Lemma \ref{lem:SBessel}. 
Interestingly, 
Bessel functions 
emerge also  in the magnetic  analogue of Hilbert's fourth problem in the plane, see 
% relation to the integral of functions along circles and discs in the plane has a long history, 
\cite{Tab} and references therein. 
\item {\bf (Nash--Moser  implicit function theorem) }
We  prove Theorem~\ref{thm:main}  constructing
zeros of $ S(a,b) $ of the form
$$
(a(\tau), b (\tau)) = \tau (\alpha, \beta ) + o(\tau) \, , \quad |\tau| \ll 1 \, , 
\quad  (\alpha, \beta ) \in V_{A_*} \, , 
$$
via a Nash--Moser implicit function theorem (below we   explain  why it is needed). 
An explicit computation shows that the linearized 
operator  % Radon magnetic transform  
$\diff S(0, 0)$ has kernel $V_{A_*}$ given in \eqref{SDABstar} and it is right-invertible, see Lemma \ref{lem:dStrivial}. 
The key difficult step to deduce Theorem~\ref{thm:main}  is Theorem \ref{thm:ri} which proves that the linearized operator $\diff S(a,b)$ admits a right-inverse in a neighborhood of $(0,0)$, satisfying tame estimates. 
\item {\bf (Invertibility of the Radon magnetic transform $ \diff S $)} 
The linearized operator $\diff S(a,b)$ can be regarded as 
a ``Radon magnetic" transform and establishing 
 tame estimates and its right invertibility 
 is much more subtle than for $\diff S(0,0) $. 
In the work of Guillemin \cite{Guillemin:1976} and Ambrozio--Marques--Neves \cite{Ambrozio:2021}, the right-invertibility of the Radon (or Radon--Funk) transform, call it  $ \mathcal F $,  
strongly  relies on the fact that the normal operator $\mathcal F \circ \mathcal F^*$ 
 is pseudo-differential. 
In our case instead $ \diff S (a,b) $ is not a Fourier integral operator,
but rather  a sum of Fourier integral operators with different phases (see Lemmata \ref{lem:tamedS} and \ref{dSadj}),
and therefore 
$\diff S \circ \diff S^*$ is {\it not}  pseudo-differential. 
We explain  below in detail  
how we overcome this major difficulty.  
Quite interestingly, a similar phenomenon appears also in 
microlocal analysis for  the X-ray transform in presence of conjugate points, as 
investigated by Holman and Uhlmann in \cite{HU18}.
\end{enumerate} 
Before discussing in detail the ideas of proof we mention that the Nash--Moser 
implicit function theorem also finds application to other related inverse problems on Riemannian surfaces as in \cite{H76,BP21}.

{\bf The action functional $ S $.} As already mentioned, the first step to apply the implicit function theorem is to characterize Zoll systems close to $(A_*,\mathrm{id}_\T)$ as zeros of a suitable {\it action functional} $S$,  originating 
from the Hamiltonian formulation  of \eqref{e:maggeo} developed in \cite{Asselle:2020}. 
 In order to write down $S$, let us recall that the equations of motion for the magnetic geodesics of $(A,B)$ 
  read 
\begin{equation}
\dot x = \cos \varphi \, , \qquad
 \dot y = \frac{\sin \varphi}{A(x)} \, , \qquad  \dot \varphi = -  \frac{B'(x)}{A(x)} - \frac{A'(x)}{A(x)} \sin \varphi \, ,
\label{eqmotion}
\end{equation}
where $\varphi \in \T$ denotes the angle formed by a unit tangent vector at $(x,y)$ with $\partial_x$.
The integrable magnetic system induced by $(A,B)$ has the 
(circle-valued) first integral 
\begin{equation}\label{firstintegral}
I_{A,B} : \T\times \T \to \T \, , \qquad I_{A,B}(x,\varphi)= A(x)\sin \varphi + B(x) \, .
\end{equation}
From now on we will assume % without further mentioning it 
that % the pair 
$ (A,B) $ belongs to a sufficiently small 
$C^1$-neighborhood  
of $(A_*,{\rm id}_\T )$
and we write  
\be\label{AxBx}
A(x) = A_* + a(x) \, , \quad B(x) = x + b(x)  \, , \qquad \forall x \in \T \, , 
\ee
where $ (a, b)$  are small periodic real functions. In this case 
$\partial_x I_{A,B} = A'(x) \sin \varphi + B'(x)  $ is everywhere positive. 
Thus, for fixed $\varphi\in\T$ we can invert $I_{A,B}$ with respect to $x$ getting a function $x_{A,B}:\T\times \T\to \T$ such that 
\begin{equation}\label{IAB}
I_{A,B}(x_{A,B}(I,\varphi),\varphi) = I \, , \qquad \forall (I,\varphi)\in \T\times \T \, .
\end{equation}
Moreover, by \eqref{eqmotion}, 
we can use the angle $\varphi$ to parametrize each magnetic geodesic. Following \cite{Asselle:2020}, for each $I\in\T$ we denote by $\Delta(a,b)(I)$ the \textit{$y$-displacement} of a magnetic geodesic with constant of motion $I$ when the angle $\varphi$ makes a full turn around $\T$. The function $\Delta(a,b)\colon\T\to\R$ is 
smooth and has zero mean. 
It follows that $\Delta(a,b)$ admits a unique primitive $S(a,b) : \T \to \R $ with zero mean, which has the expression 
\begin{equation}\label{eq:S}
 S(A,B)(I):= S(a,b) (I) := 
\int_{\T} (\cos^2\varphi) \,  A(x_{A,B}(I,\varphi)) \, 
\frac{\partial x_{A,B}}{\partial I}(I,\varphi) \, \diff \varphi - \pi A_0  
\end{equation}
where $A_0:= \frac 1{2\pi} \int_\T A(x) \, \diff x$, see \cite[Lemma 3.7]{Asselle:2020}. Thus, $I$ is a critical point of $S(a,b)$ if and only if the $y$-displacement is zero, that is, if and only if the magnetic geodesics with constant of motion $I$ are periodic. Moreover, at critical points the function $S(a,b)$ coincides with the Hamiltonian action of the periodic orbit.  
Thus, we arrive at the following result:

\begin{center}
{\it The magnetic system $(a,b)$ is Zoll  if and only if $S(a,b)\equiv0$}. 
\end{center}

{\bf The Nash-Moser implicit function theorem.} 
Since 
$  S(0,0) = 0$,
it is natural to look for  zeros  of $S$ in a neighborhood of 
$(0,0)$ by means of an implicit function theorem.  
Let  $H^s(\T) := H^s(\T,\R)$ denote the Sobolev space of $ 2 \pi $-periodic 
real-valued functions.  For any  $ s \geq \frac 72 $, 
  the action functional $ S $ continuously extends to 
\[
S\colon  \big( H^{s-\frac12}(\T)\times H^{s-\frac 12}(\T) \big)  \cap {\mathcal U} \to H^{s}_0(\T) \, , 
\quad (a,b) \mapsto S(a,b) \, ,  
\]
where $\mathcal U $ is a sufficiently small neighborhood of $(0,0) $
in $ H^{3} (\T)\times H^{3}(\T) $ and $H^{s}_0(\T) $
is the subspace  of $  H^{s}(\T) $ of 
functions with zero mean, see
Lemma \ref{lem:tameS}. Furthermore,  the differential  
\[
\diff S(0,0) : H^{s-\frac 12}(\T)\times H^{s-\frac 12} (\T)\to H^{s}_0(\T) \, , \qquad \forall s \in \R \, , 
\]
is a bounded surjective operator 
and
$ \ker \diff S (0,0)$ equal to the $H^{s-\frac 12}$-closure of the subspace $V_{A_*}$ defined in \eqref{SDABstar}, see Lemma \ref{lem:dStrivial}. In particular, surjectivity follows from the remarkable fact that the Bessel function $J_1$ and its derivative $J_1'$ do not have common zeros, as $J_1$ satisfies a second-order linear ordinary differential equation, see Lemma \ref{lem:J1}.

On the other hand,  since\[
\diff S(a,b) \colon H^{s-\frac12}(\T)\times H^{s-\frac 12}(\T)\to H^{s}_0(\T)
\] 
is a bounded operator only if $(a, b) \in H^{s+\frac 12}(\T)$, for $s\geq \frac 52$ (see Lemma \ref{lem:tamedS}), the classical implicit function theorem can not be used to 
deduce Theorem \ref{thm:main}. 
The origin of this type of problems is that $S(a,b)$ in  \eqref{eq:S} 
contains composition terms like $ A \circ x_{A,B}$, which lose regularity when differentiated with respect to $(a,b)$. 
Thus, it is natural to try to apply a Nash--Moser iteration scheme to compensate for this loss of regularity in $ (a,b) $. The key step 
is to construct, for any $(a,b)$ close to $(0,0)$, 
a right inverse of the differential $\diff S(a,b)$  satisfying the so-called tame estimates \eqref{Rabfinal}.
This is the content of the next result which is the main analytic ingredient to prove Theorem \ref{thm:main}. 
\begin{thm}\label{thm:ri}
	{\bf (Right inverse of $ \diff S (a,b) $)}
	There exists $\delta>0$ such that,  for any
	$ \| (a,b) \|_{6} < \delta $,  the operator
	$ \diff S (a,b) $ admits a right inverse  $R(a,b)  $ satisfying the tame estimates, for any $ s \geq \tfrac52 $,  
	\be\label{Rabfinal}
	\| R(a,b)  \gamma \|_s 
	\leq C (s) \|  \gamma \|_{s+\frac12} +  C(s) \| (a, b) \|_{s+\frac72} \|  \gamma  \|_{3}  \, , 
	\quad \forall \gamma \in H^{s+\frac12}_0(\T)  \,,
	\ee
	where $C(s)$ are positive constants.
\end{thm}

{\bf Ideas of proof of Theorem \ref{thm:ri}.}  Even though $\diff S(0,0)$ has a right inverse, the right-invertibility of $\diff S(a,b)$ does not follow by a direct perturbative argument because 
$\diff S(a,b)$ is not close in operator norm to $\diff S(0,0)$, cfr.~\eqref{dSa-dS0small}. The problem originates again from the term $A\circ x_{A,B}$ in \eqref{eq:S} whose differential is not continuous in $(a,b)$, no matter how regular $(a,b)$ is. The phenomenon is typical of Fourier integral operators, which involve composition terms, and in fact $\diff S(a,b)$
(see  \eqref{diffSab},  \eqref{eq:E*App})
and its $L^2$-adjoint $\diff S(a,b)^*$ (see \eqref{eq:dS*},
\eqref{eq:EApp}) 
are sum of Fourier integral operators with different phases. 

To better explain this problem and how we circumvent it to construct 
a right inverse of $\diff S(a,b)$, let us first analyze 
the example of a composition operator 
\[
T_p\colon H^s(\T)\to H^s(\T) \, ,\quad (T_p\phi)(x) :=\phi(x+p(x)) \, ,\quad \forall\,x\in\T  \, ,
\]
where $ p\in H^s(\T) $. 
The operator $T_0 $ 
is the identity. On the other hand,
\[
(T_p-T_0)\phi(x)=\left(\int_0^1\phi'(x+ \tau p(x))\diff \tau \right) \,  p(x) \, , 
\]
which is small only as operator from $H^{s+1}(\T)\to H^s(\T)$ due to the appearance of the derivative of $\phi$. Therefore, the invertibility of $T_p$ does not follow by a direct perturbative argument. On the other hand, if $p$ is $C^1$-small, the map $x\mapsto x+p(x)$ is a diffeomorphism of $ \T $ 
with inverse $ y \mapsto y + \breve p(y)$, see Lemma \ref{lem:tamecomposition}. Therefore $T_p$ is indeed invertible with inverse given by $T_{\breve p}$. The key to transfer this type of argument to more general situations is to note that $T_{\breve p}$ is related to the $L^2$-adjoint $T_p^*$  by a multiplication operator:
\[
T_{p}^*\colon H^s(\T)\to H^s(\T) \, ,\qquad T_p^*\psi=T_{\breve p}\Big(\frac{1}{1+p'}\, \psi \Big)\,.
\]
Therefore
\[
(T_p^*\circ T_p)\phi=(1+\breve p')\, \phi \, ,\qquad
(T_p\circ T_p^*)\psi=\frac{1}{1+p'}\, \psi\,.
\]
The operators $T_p^*\circ T_p$ and $T_p\circ T_p^*$ are pseudo-differential (actually, multiplication operators), and can therefore be inverted by a perturbative Neumann series argument, since they depend continuously on $p$ and are invertible for $p=0$. Hence, also $T_p$ is invertible with
\[
T_p^{-1}=T_p^*(T_p\circ T_p^*)^{-1}=(T_p^*\circ T_p)^{-1}T_p^*.
\]

Following the same pattern in our situation, the key step in the proof of Theorem \ref{thm:ri} is to show that, if $\|(a,b)\|_6<\delta$ is sufficiently small, the normal operator $M(a,b):=\diff S(a,b)\circ \diff S(a,b)^*$ is invertible and satisfies the tame estimates \eqref{Mab-1}. {However, in the present case $M(a,b)$ is {\it not} pseudo-differential
($ \diff S (a,b)$ is a sum of  Fourier integral operators
with different phases). As noted above, this is a remarkable difficulty and represents a significant difference with respect to Guillemin's situation \cite{Guillemin:1976} as well as in 
Ambrozio--Marquez--Neves \cite{Ambrozio:2021}, where the analogous operators, defined starting from the Radon transform of a metric $g$ close to $g_0$, are actually pseudo-differential. The key observation to overcome this problem is to notice that $M(a,b)$ is still pseudo-differential at the highest order so that it is enough to prove
invertibility of its restriction on subspaces of functions supported on high Fourier modes. The precise argument is carried out in Section~\ref{dSdSstar} and involves the following steps: 

\noindent
{\bf 1)}  First, we split $u=\Pi_Lu+\Pi_Ru $ % \in H^{s+\frac12}_0(\T)$ 
into \textit{low} and \textit{high} Fourier modes where 
\[\Pi_L u:= \sum_{|j|\leq N} \widehat{u}(j) e^{ijx}, \quad \Pi_R u:= \sum_{|j|>N} \widehat{u}(j) e^{ijx},\]
for $N\in \N$ large to be determined.  By the resolvent identity Lemma \ref{lem:resol}, it is enough to prove the invertibility of the operator $\tilde M_R^R$ in \eqref{ARRsmall}, which is the restriction $\Pi_R M(a,b) \Pi_R$ up to a small finite-rank operator (see Lemma \ref{lem:coupling}).
\\[1mm]
{\bf 2)} 
We  then decompose 
\[
\tilde M_R^R=\mathcal D+\mathcal N = \mathcal D( \text{Id} + \mathcal R) \, ,
\qquad \mathcal R = {\mathcal D}^{-1} \mathcal N \, , 
\]
as in \eqref{ARRNlem},  \eqref{TabDabNab},
where $\mathcal D$ is an invertible diagonal operator, see Lemma \ref{lemD}, and 
$ \mathcal N, \mathcal R $ are off-diagonal ones.
Our goal is to show that $\mathcal R$ is a bounded operator which satisfies,
 for any $ s $ larger than a threshold $ s_1 \geq 1 $, 
 \textit{asymmetric} tame estimates with \textit{smallness} as (cfr.~\eqref{tameR1R2R3}) 
\begin{equation}
\label{eq:Rgamma}
\| \mathcal R \gamma\|_s \leq C(s_1,N) \|\gamma\|_{s} + C(s,N) \|\gamma\|_{s_1} \, , \quad \forall s\geq s_1 \, ,
\end{equation}
where the \textit{asymmetry} refers to the fact that the constant $C(s_1,N)$ does not depend on $s$ and the \textit{smallness} refers to the fact that $C(s_1,N)$ can be taken arbitrarily small if $N$ 
is large enough and $\|(a,b)\|_6$ is sufficiently small. In view of \eqref{eq:Rgamma}, a direct perturbative Neumann series argument 
implies the invertibility of $\tilde M_R^R$ and tame estimates for the inverse (Lemma \ref{lem:invertame}). 
\\[1mm]
{\bf 3)}  We are left to prove the  key asymmetric tame estimates \eqref{eq:Rgamma}. Pseudo-differential operators always satisfy asymmetric tame estimates, see Lemma \ref{lem:tamepseudodiff}. 
However, as mentioned above, the operator 
$\mathcal R$ is not pseudo-differential
but just a sum of Fourier integral operators. It turns out that composition operators, and a fortiori Fourier integral operators,  satisfy 
only standard tame estimates} as \eqref{eq:tamecomposition}
 (see Remark \ref{r:nona}). To resolve this issue, in view of the decay properties of Bessel functions we decompose 
 \[
 \mathcal R= \mathcal R_1 + \mathcal R_2 + \mathcal R_3 \, ,
 \]
  where the highest order term $\mathcal R_1$ is pseudo-differential, and $\mathcal R_2$ and $\mathcal R_3$ are lower order  operators: 
\begin{itemize}
\item 
$\mathcal R_1$ satisfies asymmetric tame estimates with smallness
because % $\mathcal R_1$ has 
its matrix elements have  off-diagonal decay, see Lemma \ref{decayR1}; 
\item $\mathcal R_2$ is a small finite rank operator coming from the coupling between low and high Fourier modes in the resolvent identity decomposition and thus it
satisfies the regularizing  estimates of Lemma \ref{lem:restireso}; 
\item  $\mathcal R_3$ is not pseudo-differential 
 but a sum of Fourier integral operators 
more regularizing than $ {\mathcal R}_1$. 
By such  additional regularity  
we can find $ N $ large enough to achieve smallness and use % and we can use 
the interpolation inequality \eqref{asintpo} to 
 turn the standard tame estimates of Fourier integral operators into asymmetric ones, see Lemmata \ref{restoE1} and \ref{lemR2}.
\end{itemize}

{\bf Structure of the paper.} 
 In Section \ref{sec:func} we introduce the analytical tools  to prove 
 tame estimates in Sobolev spaces for the action functional  $ S $ 
 (Section \ref{sec:S}), its differentials and their adjoints (Section  \ref{sec:diffS}). 
%In Section \ref{sec:S} we provide the Fourier expansion of 
%$ S $ and show that it satisfies tame estimates.  
%In Section \ref{sec:diffS} we compute 
%$\diff S $,
%$ \diff^2 S $ as well as their adjoints and prove that they fulfill tame estimates. 
 Section \ref{dSdSstar}  proves the key result
about the invertibility of $\diff S \circ \diff S^* $ in a neighborhood of $(0,0)$
 which 
 implies Theorem \ref{thm:ri}. % with tame estimates. 
Section \ref{sec:NM} is finally devoted to the proof of 
Theorem \ref{thm:main}. % using as a black-box abstract results of \cite{Hamilton,Ambrozio:2021}.

\vspace{2mm}

\textbf{Acknowledgments.} L.A. and G.B. are partially supported by the Deutsche Forschungsgemeinschaft (DFG) through the SFB/TRR
191 (Project-ID 281071066) ``Symplectic Structures in Geometry, Algebra and Dynamics''. G.B.~is
partially supported by the Deutsche Forschungsgemeinschaft 
under Germany's Excellence Strategy
EXC2181/1 - 390900948 (the Heidelberg STRUCTURES Excellence Cluster) and the Research Training Group RTG 2229 - 281869850 (Asymptotic Invariants and Limits of Groups and Spaces). M.B.~is supported by PRIN 2020 (2020XB3EFL001) ``Hamiltonian and dispersive PDEs". M.B. thanks Pietro Baldi 
for useful conversations. The authors are grateful to the anonymous referees for the careful reading of the manuscript.

{\bf Notation.}
We denote 
by $ \N := \{1, 2, \ldots \} $ the positive integers and $ \N_0  := \{ 0 \} \cup \N $. 
The notation $ a \lesssim_s b $ means that $ a \leq C(s) b $ for some positive constant 
$ C(s) > 0 $.  Notation $ a \sim_s b $ means that $ C_1(s)  b \leq a \leq C_2 (s)  b  $
for $ 0 < C_1 (s) < C_2 (s) $. Along the paper $ C(s) $ denote 
different positive constants non-decreasing in $ s $. If $y$ is a real number, we denote $\langle y \rangle:=\max\{1,|y|\}$.

\section{Functional setting}\label{sec:func}

In this 
 section we present basic properties of Sobolev spaces, 
 the notion of tame operators,  matrices with off-diagonal decay, 
 pseudo-differential operators and their symbols, 
composition and Fourier series/integral operators.

\vspace{1mm}

{\bf Sobolev spaces.}  
The Hilbert space 
$L^2(\T) := L^2(\T,\C)$ 
is equipped with the usual 
scalar product 
\be\label{uvL2}
(u,v)_{L^2} := \frac{1}{2\pi}\int_{\T} u(x) \overline{v(x)}  \, \diff x \, , 
\ee  
and we denote $ \{ e_j \}_{j \in \Z} $, where $ e_j := e_j (x) := e^{ i j x} $, its canonical 
orthonormal  basis. 
The product space $ L^2 (\T) \times L^2 (\T) $ is endowed with the usual scalar product $( (u_1,u_2),(v_1,v_2))_{L^2\times L^2} := (u_1,v_1)_{L^2} + (u_2,v_2)_{L^2}$.

For any $s\in \R$,  we consider the Sobolev space
\be\label{Sobos}
H^s (\T)  := H^s (\T,\C)  := \Big\{ 
u(x) = \sum_{j \in \Z} \widehat u(j) e_j  \ : \ \| u\|_s^2 := \sum_{j\in \Z} \langle j \rangle^{2s} |\widehat{u}(j)|^2  <+ \infty 
\Big\}
\ee
where $ \widehat u(j) = (u, e_j)_{L^2} $ are the Fourier coefficients of the $ 2 \pi $-periodic 
function $ u(x) $ .
We have $ H^0 (\T )  = L^2 (\T) $. We denote by $ H^s_0 (\T ) $ the subspace of functions in $H^s (\T) $
with zero mean. We will work with the subspace $H^s (\T, \R) $ of
real-valued functions in \eqref{Sobos} and with real operators acting between such spaces.
A  function $u $ in $ H^s(\T,\R)$ is characterized by 
$ \widehat u(-j) = \overline{\widehat u(j)} $, for any $ j \in \Z $. 

The Sobolev norms $ \| \cdot \|_s $ satisfy the following interpolation estimates, see e.g.~\cite[Lemma 4.5.4]{Berti:2013}: let $ \alpha \leq a \leq b \leq \beta $ such that
$ \alpha + \beta = a + b $ then 
\be\label{sumabab}
\| u \|_{a} \| u \|_b \leq \| u \|_{\alpha} \| u \|_{\beta} \, . 
%\| u \|_{\lambda s_1+ (1-\lambda) s_2} \leq \| u \|_{s_1}^{\lambda} \| u \|_{s_2}^{1-\lambda} 
\ee
The Sobolev norms  $ \|\cdot \|_s $ also satisfy  the following interpolation estimates, see 
\cite[Lemma 2.2]{Berti:2020}:  
let $ s_1, s_2 \geq 0 $ and $ q_1,q_2 > 0 $, then  for any $ 	\epsilon > 0 $ there is a constant $ C(\epsilon ) > 0 $ such that 
\be\label{asintpo}
\| u \|_{s_1+q_1} \| v \|_{s_2+q_2} \leq 
\epsilon \| u \|_{s_1+q_1+q_2}  \| v \|_{s_2} + 
C(\epsilon) \| u \|_{s_1}  \| v \|_{s_2+q_1+q_2} \, . 
\ee

For $ s > 1 / 2 $, the Hilbert space $H^s(\T)$ compactly embeds in $C^0(\T)$, 
the spaces $H^s (\T)  $ form an algebra with respect to the product of functions 
and the following asymmetric  tame estimates  hold (see e.g. \cite[Section 4.5]{Berti:2020b}):
for any  $ s \geq s_0>1/2 $ and any $u,v\in H^s(\T)$, 
\begin{equation}\label{eq:tameproduct}
\|uv\|_s  
 \leq C(s) \|u\|_s \|v\|_{s_0}+  C(s_0) \|u\|_{s_0} \|v\|_s 
\end{equation}
where the  constants $ C(s) > 0 $ are non-decreasing in $ s $. 
We also remind the inequality (\cite[Lemma 4.5.2]{Berti:2013}): 
for $ s_0 > 1 / 2 $ and any $ 0 \leq s \leq s_0 $,
any $ u \in H^{s_0}(\T) $, $ v \in H^{s} (\T) $, 
\be\label{intermba}
\| u v \|_s \leq C(s_0) \| u \|_{s_0} \| v \|_s \, .  
\ee
Throughout the paper we will fix the threshold (see Lemma \ref{lem:tamecomposition}) 
\[
s_0=1 \, .
\] 
{\bf Tame operators.} Given $ \sigma \geq 0 $, 
we say that a linear operator $ M $ acting on the scale of Sobolev 
spaces  $ H^s(\T) $ is $ \sigma $-tame
if there is $s_1\geq s_0$ such that, for any $  s \geq s_1 $,  
it satisfies 
\be\label{sigmatame}
\| M u  \|_s \leq C_M (s_0) \| u \|_{s+\sigma} + C_M (s) \| u \|_{s_0+\sigma}
\ee
where the functions $ s \mapsto C_M (s) \geq 0 $ are non-decreasing in  $s $.
We say that $ C_M (s) $ are tame constants for the operator $ M  $.  
The composition $ M_1 M_2 $ of a $ \sigma_{M_1}$-tame operator $ M_1 $
and a $ \sigma_{M_2}$-tame operator $ M_2 $ 
 is a $ (\sigma_{M_1}+  \sigma_{M_2}) $-tame operator.
 
 % with tame constants
%\be\label{CM1M2}
%C_{M_1M_2} (s) \leq 
% {2 C_{M_1} (s)} C_{M_2} (s_0+\sigma_{M_1}) +%
%C_{M_1} (s_0) C_{M_2} (s+\sigma_{M_1})  \, . 
%\ee
We prove a simple lemma about the invertibility of tame operators. 
\begin{lem}{\bf (Tameness of the inverse)}\label{lem:invertame}
Assume that $  {\mathcal R} $ is a $ 0$-tame operator, namely that \eqref{sigmatame} holds
with $ \sigma = 0 $.
Then, if 
$ C_{\mathcal R}  (s_1)  < 1 / 4 $, the 
linear operator $ {\rm Id} +  {\mathcal R} $ is invertible and its inverse 
satisfies the tame estimates
\begin{align}
& \|  ({\rm Id} +  {\mathcal R})^{-1} v \|_{s_1} \leq 2\| v \|_{s_1} \label{1+Rs0} \\
& \| ({\rm Id} +  {\mathcal R})^{-1} 
v \|_s \leq 2  \| v \|_{s} +  4C_{\mathcal R}  (s)  \| v \|_{s_1} \, ,
\qquad \forall s \geq s_1 \, . \label{1+Rstame}
\end{align}
\end{lem}

\begin{proof}
The invertibility of the operator $ {\rm Id} +  {\mathcal R} : H^{s_1}(\T) \to H^{s_1}(\T)$ 
and the bound \eqref{1+Rs0} directly follows  
by $ \|   {\mathcal R} u \|_{s_1} \leq 2 C_{\mathcal R}  (s_1) \| u \|_{s_1} <   \| u \|_{s_1} / 2 $ 
(which is \eqref{sigmatame} for $ M = {\mathcal R} $ at $ s = s_1 $). 
Then we prove \eqref{1+Rstame}. 
We have 
\[
u =  ({\rm Id} +  {\mathcal R})^{-1} v \quad \Longleftrightarrow \quad ({\rm Id} + {\mathcal R} ) 
u =  v \quad \Longleftrightarrow \quad 
u =  v -  {\mathcal R} u \, . 
\]
Thus, for any $ s \geq s_1 $,  
using \eqref{sigmatame} for $ {\mathcal R}  $ and \eqref{1+Rs0}, we get 
\[
\| u \|_s 
 \leq \| v\|_s + \|  {\mathcal R} u \|_s  \leq  \| v\|_s + 
C_{\mathcal R} (s)  \| u \|_{s_1} +  C_{\mathcal R} (s_1)  \| u \|_{s}
 \leq  \| v\|_s +  C_{\mathcal R} (s) 2  \| v \|_{s_1} +  C_{\mathcal R} (s_1)  \| u \|_{s}
\]
and therefore, since $ C_{\mathcal R}  (s_1) < 1 / 2 $, 
\[
\tfrac12  \| u \|_s  \leq (1 - C_{\mathcal R} (s_1)) \| u \|_s  \leq \| v\|_s +  2 C_{\mathcal R}  (s)  \| v \|_{s_1} 
\]
implying \eqref{1+Rstame}. 
\end{proof}

{\bf Matrix representation of linear operators.} 
A linear operator 
$ M :  L^2 (\T) \to L^2 (\T) $ 
can be represented with respect to the exponential basis $ \{ e_j \}_{j \in \Z} $ 
in the matrix form 
\[
M u = \sum_{k \in \Z} \Big( \sum_{j \in \Z} M^j_k \widehat{u}(j) \Big) e_k \, , \qquad
M^j_k := ( M e_j, e_k )_{L^2} \, . 
\]
In what follows we shall identify  the operator $ M $ with the matrix $ (M^j_k)_{j,k\in \Z} $.
If $ M $ is a diagonal operator, i.e. $ M_k^j = 0 $ for any $ j \neq k $, we also write 
$ M = {\rm diag} (M_j^j )$. 

Following \cite[Appendix B]{Berti:2013} we define the $ s $-decay norm  of $M$ by
\be\label{sdecay}
|M|_s^2 := \sum_{m \in \Z}  \Big( \sup_{j-k=m} |M^j_k| \Big)^2 \langle m \rangle^{2s}.
\ee
Such a norm is designed to estimate the off-diagonal decay of matrices similar to the Toeplitz matrices which represent the multiplication operator for a Sobolev function. Indeed
the $ s $-decay norm of the multiplication operator by a $ 2 \pi $-periodic function $ p(x) $ is
\be\label{toplitzpro}
|p|_{s} = \| p \|_s \, . 
\ee
The  $ s $-decay norm satisfies the following asymmetric 
interpolation estimates (see e.g. \cite[Appendix B1]{Berti:2020}), for any $ s \geq s_0 > \tfrac12 $,   
\be\label{M1M2tameas}
| M_1 M_2 |_s \leq C(s_0) | M_1|_{s_0} | M_2|_{s} + C(s) | M_1|_{s} | M_2|_{s_0} \, . 
\ee
It also controls the action on Sobolev spaces, for any $ s \geq s_0 > \tfrac12  $,  
\be\label{tameMsa}
\| M u  \|_s \leq C(s_0) | M|_{s_0} \| u \|_{s} + C(s) | M |_{s} \| u \|_{s_0} \, . 
\ee

{\bf Finite rank operators.} 
Given $ N \in \N $, we  say that a linear operator $ {\mathcal R}  $ is ``finite rank" if  it has the form 
\be \label{defFR}
{\mathcal R}   = \sum_{|j| \leq N} ( \cdot, g_j )_{L^2} \chi_j 
\ee
where $ g_j $, $ \chi_j $   are $ C^\infty (\T) $ 
smooth functions. Such an operator is infinitely many times regularizing, in particular, for any $ s \geq 0  $,  
\be\label{Rure}
\| {\mathcal R} u \|_s \leq   \Big( \sum_{|j| \leq N}\| g_j \|_0\|  \chi_j  \|_s\Big) \| u \|_0 \, . 
\ee
The adjoint of the operator ${\mathcal R}  $ defined 
 in \eqref{defFR} with respect to the $ L^2$-scalar product \eqref{uvL2} is 
\be\label{RFRadj}
{\mathcal R}^*  = \sum_{|j| \leq N} ( \cdot, \chi_j  )_{L^2} g_j \, . 
\ee

{\bf Pseudo-differential operators.} 
For any function $ h : \T \times \Z \to \C $ 
we define the ``quantized pseudo-differential" operator
\be\label{pseudoop}
h(x,D) u  := \sum_{j\in \Z} h(x,j) \widehat{u}(j) \, e_j   
\ee
where $ D = \frac{1}{i} \partial_x $ is the H\"ormander derivative.
\begin{rem}
An operator $ h(x,D) $ as in \eqref{pseudoop} is not pseudo-differential 
unless we require some properties about the regularity in $ x $
and the growth in $ j $ of $ h(x,j)$  and its discrete derivatives 
with respect to $ j $.  However, we still find convenient to refer to $ h(x,D)$ 
as a pseudo-differential operator and call $ h(x,j) $ its symbol. 
 If $ h(x,j) = h(j) $ does not depend on $ x $ we say that $ h(D) $ is a Fourier multiplier. 
\end{rem}

The action of an operator 
$ h(x,D) $ as in \eqref{pseudoop} on Sobolev spaces 
 is controlled by the following norm. 
For any  $s , \sigma \in \R $,   we define 
\be\label{normahs}
|h|_{\sigma,s} := \sup_{j\in \Z}\  \langle j\rangle^{-\sigma} \|h(\cdot,j)\|_{s} \, .
\ee
If $|h|_{\sigma,s}$ is finite for any 
$ s \in \R $, then we say that the symbol $h(\cdot, j) $ \textit{has order} $\sigma$.
The reason is the following  lemma   about the action 
on Sobolev spaces of an operator as in \eqref{pseudoop}
that is a direct consequence of  \cite[Lemma 2.21]{Berti:2020}, writing
\[ 
  h(x, D) =  \tilde h (x, D) \langle D \rangle^\sigma  \, , 
  \qquad  \tilde h (x,j) :=  h(x, j) \langle j \rangle^{-\sigma} \, , 
\]
and the identities
\begin{equation}\label{hsigmah0}
|h|_{\sigma,s} = |
\tilde h |_{0,s} \, , \qquad
\| \langle D \rangle^\sigma u \|_s = \| u \|_{s+\sigma} \, , \quad  \forall  s \in \R \, . 
\end{equation}

\begin{lem} {\bf (Action of pseudo-differential operators)} 
\label{lem:tamepseudodiff}
Let  $\sigma \in \R $. 
For any $  s  \geq s_0>1/2$ there exist constants $ C(s) := C(s, \sigma ) > 0 $,  
non decreasing in $ s $, 
such that, for any $ u \in H^{s+\sigma}(\T)  $, 
\be\label{actiontamepse}
\|h (x,D) u \|_s \leq C(s_0) |h|_{\sigma,s_0} \| u \|_{s+\sigma} + 
C(s) 
|h|_{\sigma,s}\| u \|_{s_0+\sigma}  \, .
\ee
\end{lem}
Clearly \eqref{eq:tameproduct} is a particular case of \eqref{actiontamepse} since for a function $ h(x,j) = u(x)  $ we have $ | h |_{0,s} = \| u \|_s $ . 

We shall use the following lemma. 

\begin{lem}\label{lem:piccolosimbolo}
Let $ N \in \N $. Suppose that $ h(x,j) = 0 $ for any $ |j| \leq N $. Then, for any 
$ \sigma_1 \geq \sigma_2 $  we have
\[
| h |_{\sigma_1,s} \leq N^{\sigma_2-\sigma_1}  | h |_{\sigma_2,s} \, .
\] 
\end{lem}

\begin{proof}
In view of \eqref{normahs} and since $ h(x,j) = 0 $ for any $ |j| \leq N $, we have
\[
|h|_{\sigma_1,s} = \sup_{j\in \Z}\  \langle j\rangle^{-\sigma_1} \|h(\cdot,j)\|_{s}
=  \sup_{ |j|> N}\  \langle j\rangle^{\sigma_2 -\sigma_1} \langle j\rangle^{-\sigma_2} \|h(\cdot,j)\|_{s} \leq N^{\sigma_2 -\sigma_1} |h|_{\sigma_2,s}
\]
proving the lemma. 
\end{proof}

By Lemma 2.16 in \cite{Berti:2020}, the adjoint of $h(x,D)$ with respect to the $ L^2 $ scalar 
product \eqref{uvL2} has the form   
\be\label{adjhD}
\big ( h(x,D)\big)^* = h^* (x,D)
\ee
where the symbol  $ h^* (x,j)$
(for its explicit formula see Equation (2.31) in \cite{Berti:2020})
satisfies:  there exists $ C > 0 $ such that, for any  $s \in \R $,   
\begin{equation}\label{eq:tamekappa0}
| h^*  |_{0,s} \leq  C | h|_{0,s+s_0} \, . 
\end{equation}

{\bf Estimates of special symbols.} 
We now estimate the norm of a class of symbols used frequently throughout the paper.

\begin{lem}\label{lem:Sobosimbol} 
Let 
$\ell: \R \to \C$ be a $ C^\infty $ function   satisfying 
\begin{equation}\label{eq:1lemma1s2}
\| \ell \|_{\sigma,m} := \sup_{y \in \R} \, \langle y \rangle ^{-(\sigma-m)} |\ell^{(m)}(y)| < +\infty \, 
, \quad \forall m\in \N_0 \,, 
\end{equation} 
where $\ell^{(m)}$ denotes the $m$-th derivative of $\ell$. Let $ A_* > 0  $ and 
$ A :\T\to (0,+\infty)$ be a  $ C^\infty $-function of the form $ A (x) = A_* + a(x) $ 
with $ \| a \|_{C^0} \leq A_* / 2 $ and $ \| a \|_{1} \leq 1 $.      
Then  
\begin{equation}\label{eq:symbolh}
h:\T\times \Z\to \C \, , \quad h(x,j) := \ell (j A(x))  \, , 
\end{equation}
is a  symbol  of order $ \sigma $
satisfying, for any $s \geq 0  $, 
\begin{equation}\label{eq:2lemma2s2}
|h|_{\sigma,s} \leq C(s, A_*, \ell) 
% \Big( \max_{q= { 0,\ldots, s}} \| \ell \|_{\sigma,q}\Big) 
 (1+\|a\|_{s}) 
\end{equation} 
where $ C(s, A_*, \ell) > 0 $ are positive constants non-decreasing in $ s $. 
\end{lem}

\begin{proof}
Recalling \eqref{normahs} and \eqref{eq:symbolh}, 
the bound \eqref{eq:2lemma2s2} follows by  the next claim (for $ m = 0 $): 
\\[1mm]
{\sc Claim} : {\it for any $ m \in \N_0 $, any $ s \geq 0 $, 
there exists $ C_{s,m,A_*,\ell} > 0 $ such that }
\be\label{lemax}
\| \ell^{(m)} ( j (A(x)))\|_s \leq C_{s,m,A_*,\ell} (1 + \| a \|_s)  \langle j \rangle^{\sigma-m} \, ,
\quad \forall j \in \Z \, . 
\ee
We prove  \eqref{lemax}  by induction. 
In the sequel we use that, 
since $ \| a \|_{C^0} < A_* /2  $ then $ A(x) = A_* + a(x) $ satisfies $  \frac{A_*}{2}  < A(x) <  \frac{3 A_*}{2}  $ for any $ x \in \T $.  
\\[1mm]
{\it Initialization.} For $ s = s_0 = 1 $ we have 
\[
\begin{aligned}
\| \ell^{(m)} ( j (A(x)))\|_{1} 
& \sim \| \ell^{(m)} ( j (A(x)))\|_{L^2} + \| \pa_x \big( \ell^{(m)} ( j (A(x)))\big) \|_{L^2}  \\
& \lesssim \| \ell^{(m)} ( j (A(x)))\|_{C^0} + |j| \| \ell^{(m+1)} ( j (A(x))) \|_{C^0} \| a_x   \|_{L^2}  \\
& \stackrel{\eqref{eq:1lemma1s2}} {\lesssim_{m,s,A_*, \ell}} 
\langle j \rangle^{\sigma-m} (1 + \| a   \|_{1})
\lesssim_{m,s,A_*, \ell} 
\langle j \rangle^{\sigma-m} 
\end{aligned}
\]
because $ \| a \|_1 \leq 1 $. Thus, 
for any $m \in \N_0 $, any $ s \in [0,s_0] $ we have that   
\[ \| \ell^{(m)} ( j (A(x)))\|_{s} \leq \| \ell^{(m)} ( j (A(x)))\|_{1}   \leq C_{s,m,A_*, \ell} \langle j \rangle^{\sigma-m}  \leq  C_{s,m,A_*, \ell} \langle j \rangle^{\sigma-m} (1+ \| a \|_s ),\]
 which is 
\eqref{lemax} for any $ s \in [0,s_0] $. 
\\[1mm]
{\it Induction:} Given some integer $ k \geq s_0 = 1 $ we assume that \eqref{lemax}
holds for any $ m \in \N_0 $ and any $ s \in [0,k] $. We are going to prove \eqref{lemax}
for any $ s \in (k,k+1] $. 
We have 
%Since 
%$ \| u \|_s % = \sum_{j \in \Z} |u_j|^2 \langle j \rangle^{2s} 
%\sim_s \| u \|_{L^2} + \| \pa_x u \|_s  $ we have  
\begin{align}
 \| \ell^{(m)} ( j (A(x)))\|_{s} & \sim_s
  \| \ell^{(m)} ( j (A(x)))\|_{L^2} + \big\| \pa_x \big( \ell^{(m)} ( j (A(x)))\big) \big\|_{s-1} \notag  \\
  & \lesssim_s
  \| \ell^{(m)} ( j (A(x)))\|_{C^0} 
  + |j| \| \big( \ell^{(m+1)} ( j (A(x)))\big) a_x (x) \|_{s-1}  \notag  \\
    & \lesssim_{s,m,\ell, A_*} \langle j \rangle^{\sigma-m}
  + |j| \| \big( \ell^{(m+1)} ( j (A(x)))\big) a_x (x) \|_{s-1} \label{pripez1}
  \end{align}
by \eqref{eq:1lemma1s2} and since   $  \frac{A_*}{2}  < A(x) <  \frac{3 A_*}{2}  $. 
  To estimate the last term in \eqref{pripez1} we distinguish two cases.

\vspace{2mm}

{\it 1st case:} $ k = s_0 = 1 $. Thus $ s \in (s_0,s_0+1]$ and $ s -1 \in (s_0-1,s_0] $ 
 and 
 the inequality 
 \eqref{intermba} implies 
 that \eqref{pripez1} is bounded by  
 \[
 \begin{aligned}
  \| \ell^{(m)} ( j (A(x)))\|_{s} 
&  \lesssim_{s,m,A_*,\ell} \langle j \rangle^{\sigma-m} + 
    |j| \| \big( \ell^{(m+1)} ( j (A(x)))\big)\|_{s_0} \| a_x \|_{s-1} \\
 &   \stackrel{\eqref{lemax}} {\lesssim_{s,m,A_*,\ell}} \langle j \rangle^{\sigma-m} +   |j| \langle j \rangle^{\sigma-m-1} (1+ \| a \|_{s_0}) \| a \|_{s}  \\
&  \lesssim_{s,m,A_*,\ell} \langle j \rangle^{\sigma-m} ( 1 +  \| a \|_{s})    
 \end{aligned} 
 \]
 since $ \| a \|_{s_0} \leq 1 $. This 
  proves \eqref{lemax} for $ s \in (s_0,s_0+1] $.

\vspace{2mm}

{\it 2nd case:} $ k \geq s_0 + 1 = 2 $. For any $ s \in (k,k+1]$ we have 
 $ s - 1 \geq s_0 $  and by \eqref{pripez1} 
 and  \eqref{eq:tameproduct} we get  
  \begin{align}
   \| \ell^{(m)} ( j (A(x)))\|_{s}  
  &  \stackrel{\eqref{eq:1lemma1s2} } {\lesssim_{s,m,A_*,\ell}}
\langle j \rangle^{\sigma - m }  + 
  |j| \big\| \big( \ell^{(m+1)} ( j (A(x)))\big) \big\|_{s_0} \| a_x  \|_{s-1} +  
  |j| \big\| \big( \ell^{(m+1)} ( j (A(x)))\big) \big\|_{s-1} \| a_x  \|_{s_0} \notag  \\ 
  & \stackrel{\eqref{lemax}, \| a \|_{s_0} \leq 1 } {\lesssim_{s,m,A_*,\ell}}
\langle j \rangle^{\sigma - m }  + 
  |j|  \langle j \rangle^{\sigma - m - 1}  \| a  \|_{s} +  
  |j|  \langle j \rangle^{\sigma - m - 1} (1+ \| a \|_{s-1}) \| a  \|_{s_0+1} \,. \label{pezz2c}
  \end{align}
  Finally by \eqref{sumabab} we have $\| a \|_{s-1} \| a  \|_{s_0+1}
  \leq  \| a \|_{s} \| a  \|_{s_0}$ and, since $ \| a  \|_{s_0} \leq 1 $, 
  we deduce  by \eqref{pezz2c} 
  that $   \ell^{(m)} ( j (A(x)))   $ satisfies
 \eqref{lemax} for $ s \in (k, k+1]  $. 
 
 This concludes the inductive proof of the claim and thus of the lemma. 
\end{proof}

\vspace{1mm}

{\bf Composition operators under diffeomorphism.} Given a diffeomorphism of $ \T $,
$ x \mapsto x + p(x) $, where $ p(x)$ is a $   2 \pi  $-periodic real function, we define
$ y \mapsto y + \breve p (y)  $ its inverse diffeomorphism of $ \T $, and 
 the associated composition operators 
\be\label{defcalP}
(\mathcal P u) (x) := u(x+p(x)) \, , \quad (\mathcal P^{-1}v)(y) := v (y+\breve p(y)) \, . 
\ee
For $ k \in \N_0 $ we denote % $ \| \ \|_{C^k} $ the usual $ C^k $ norm. 
$ \| u \|_{C^k} := \max_{0 \leq m \leq k} \| u^{(m)} \|_{C^0} $ where  
$ \| u  \|_{C^0} := \sup_{x \in \T} |u(x)| $.

\begin{lem}\label{lem:tamecomposition} {\bf (Diffeomorphism of $ \T$)}
\cite[Lemma 2.30]{Berti:2020}
Assume $ \| p \|_{C^{s_0+1}} < 1 / 2 $ with $ s_0 = 1 $. Then
$ x \mapsto x + p(x) $ is a diffeomorphism of $ \T $ and
\begin{enumerate}
\item 
$ \| \breve p \|_s \leq C(s) \| p \|_s $ for any $ s \geq s_0 $;
\item for any $ s \geq s_0+1 $,  there exist constants $ C(s) > 0 $, non-decreasing in $ s $, 
such that, for any $ u \in H^s(\T)  $, 
\begin{equation}\label{eq:tamecomposition}
\|\mathcal P u \|_s + \|\mathcal P^{-1}u\|_s \leq C(s) \|u\|_s + 
C(s_0) \| p \|_s \|u\|_{s_0+1}  \, .
\end{equation}
\end{enumerate}
\end{lem}

\begin{rmk}\label{r:nona}
Composition operators do not satisfy asymmetric tame estimates in which the constant in front of $\|u\|_s$ is independent of $s$. Indeed, for any $ s \in \N $, the $ s $-th derivative of $\mathcal Pu$ is 
\[ (\mathcal Pu)^{(s)}(x) = u^{(s)}(x+p(x))  (1+ p'(x))^s + ...   \]
and
the $ L^2 $-norm  
$ \| u^{(s)}(x+p(x)) 
(1+ p'(x))^s\|_{L^2} 
$
provides a contribution to \eqref{eq:tamecomposition}
with constant $C(s)$  exponentially increasing in $s$. \end{rmk}

\begin{rmk}
Asymmetric tame estimates  for   the composition operator, as well as of any FIO, 
can be obtained
by allowing an arbitrarily small loss of derivatives. 
More precisely, for any 
$\delta>0 $, 
by  \eqref{eq:tamecomposition} and the interpolation inequality \eqref{asintpo} with $ v = 1 $, $ s_1 = s_0 + 1 $, $ q_1 = s- s_0 - 1 $, $ q_2 = \delta $,  taking $\epsilon=1/C(s)$,  
 there are constants $C(s)>0$, non-decreasing in $s$,  such that,
for any $ s \geq s_0 + 1 $, $s\in \N$, and $ u \in H^s(\T) $, 
\[
\|\mathcal P u\|_s % + \|\mathcal P^{-1}v\|_s
 \leq \|u\|_{s+\delta} + C(s) (1+\| p \|_s) \|u\|_{s_0+1}  \, .
\]
Observations of this kind 
allow us to use the Nash--Moser implicit function theorem in the smooth category, see Lemma \ref{lemR2}.
\end{rmk}

The adjoint  with respect to the $ L^2 $-scalar 
product \eqref{uvL2} of the composition operator $ {\mathcal P} $ in \eqref{defcalP} is 
\be\label{adjcompo}
\mathcal P^* =\mathcal P^{-1}\circ \frac{1}{1+p'}  
\ee
where  $ \frac{1}{1+p'}$ is the multiplication operator by the function $ \frac{1}{1+p'} $.  

\vspace{1mm}

{\bf Fourier Series/Integral Operators.}
We shall encounter  also  \textit{Fourier Series/Integral Operators} (FIO for short)  of the form 
\begin{equation}\label{eq:EApp}
Eu := \sum_{j\in \Z} h(x,j) \widehat{u}(j) e^{ i j (x+ p(x))}   \,. 
\end{equation}
If $ p (x) = 0 $ then $ E $ reduces to the pseudo-differential operator $ h(x,D)$ in \eqref{pseudoop}, 
whereas for $ h(x,j) = 1 $
it reduces to the composition operator $ {\mathcal P} $ in \eqref{defcalP}. 
In general, note that we may write 
\begin{equation}\label{eq:mathcalEcomposition}
E =  \mathcal P \circ \breve h (x,D) 
\end{equation}
as the composition of the diffeomorphism $ {\mathcal P} $ in \eqref{defcalP} and the 
operator $ \breve h (x,D) $
where, for any $ j \in \Z $,  
\begin{equation}\label{eq:htilde}
\breve h(y,j):= (\mathcal P^{-1}h)(y,j) \, . 
\end{equation}
The adjoint of $ E $ in \eqref{eq:EApp} with respect to the 
$ L^2 $ scalar product in \eqref{uvL2} is 
\begin{equation}\label{eq:E*App}
E^*u = \sum_{k\in \Z} \left ( \int_\T u(x) \overline{h(x,k)} e^{-i k (x+p(x))}\, \diff x \right )  e_k \, .
\end{equation}
Notice that, in view of \eqref{eq:mathcalEcomposition} %\eqref{adjhD} 
and
 \eqref{adjcompo}, we can also write the adjoint   $ E^* $ as
\begin{equation}\label{eq:mathcalE*}
 E^*= % \big (\breve h (x,D)\big )^* \circ \mathcal P^* =
%\breve h^* (x,D)
\big (\breve h (x,D)\big )^*   \circ \mathcal P^{-1}\circ \frac{1}{1+p'} 
\end{equation}
which is the composition of a multiplication operator, 
a composition operator and a pseudo-differential one, cfr.~\eqref{adjhD}. 

We now prove tame estimates in Sobolev spaces of FIO operators and their adjoints.

\begin{lem}\label{lem:Eusmall}
{\bf (FIO)}
Let $ \sigma \in \R $ 
and   assume $ \| p \|_{C^{s_0+1}} < 1 / 2 $ with $s_0 = 1 $. 
Then 
the FIO operator $ E $  in \eqref{eq:EApp} satisfies the tame estimates, for any $ s \geq s_0 + 1 $,
\be\label{stimaEu}
\| E u \|_s \leq   C(s) |h|_{\sigma,s_0+1}  \| u \|_{s+\sigma}  +
C(s) 
\Big( |h|_{\sigma,s} + \| p \|_s |h|_{\sigma,s_0+1}  \Big) \| u \|_{s_0+ \sigma+1} 
\ee
where 
$ C(s) := C(s, \sigma)  > 0 $  are constants  non-decreasing in $ s $.  
\end{lem}

\begin{proof}
By  \eqref{eq:mathcalEcomposition} we have 
$ E =   \mathcal P \circ \breve h (x,D)  $ where $ \breve h  $ in \eqref{eq:htilde}
satisfies for any $ s \geq s_0 + 1 $  
\be\label{sthbreve}
\begin{aligned}
| \breve h |_{\sigma,s} 
 \stackrel{\eqref{normahs}} {:=} \sup_{j \in \Z}  \, \langle j \rangle^{-\sigma} 
\| (\mathcal P^{-1}h)(\cdot ,j) \|_s &  \stackrel{\eqref{eq:tamecomposition}} \leq  
\sup_{j \in \Z}  \, \langle j \rangle^{-\sigma} 
 \Big( C(s) \| h (\cdot ,j) \|_s + C(s_0) \| p \|_s \| h (\cdot, j)\|_{s_0+1} \Big) \\
 &  \leq   C(s) | h |_{\sigma,s} + C(s_0) \| p \|_s | h |_{\sigma,s_0+1} \, . 
 \end{aligned}
\ee
In particular $ | \breve h |_{\sigma,s_0+1} \lesssim_{s_0} | h |_{\sigma,s_0+1} $ 
being  $ \| p \|_{C^{s_0+1}} \leq 1/ 2 $.  
Then the estimate \eqref{stimaEu} for $ E =  \mathcal P \circ \breve h (x,D) $ 
 follows by \eqref{eq:tamecomposition},  for any $ s \geq s_0 + 1 $, 
\begin{align} 
\| E u \|_s 
& \leq C(s) \| \breve h (x,D) u \|_s + C(s_0) \| p \|_s \| \breve h (x,D) u \|_{s_0+1}  \notag \\
& \stackrel{\eqref{actiontamepse}} 
\leq C (s) \Big[ |\breve h|_{\sigma,s}\| u \|_{s_0+\sigma} +  | \breve h|_{\sigma,s_0}
\| u \|_{s+\sigma}  \Big]+ C(s_0) \| p \|_s 
 \Big[ |\breve h|_{\sigma,s_0+1}\| u \|_{s_0+\sigma} +  | \breve h|_{\sigma,s_0}
\| u \|_{s_0+1+\sigma}  \Big] \, . 
  \notag \\
  & \stackrel{\eqref{sthbreve}} 
\leq  C (s) \big( | h |_{\sigma,s} + \| p \|_s | h |_{\sigma,s_0+1} \big) 
\| u \|_{s_0+\sigma} +   C (s) |  h|_{\sigma,s_0+1}
\| u \|_{s+\sigma} + C(s_0) \| p \|_s  | h|_{\sigma,s_0+1}
\| u \|_{s_0+1+\sigma}   
  \notag 
 \end{align}
proving \eqref{stimaEu}. 
% Lemma \ref{lem:tamepseudodiff} and \eqref{sthbreve}. 
\end{proof}

The adjoint of the operator $ E $ in \eqref{eq:EApp} satisfies the following 
tame estimates.

\begin{lem}\label{lem:Estargene}
{\bf (Adjoint of FIO)}
Let $ \sigma \in \R $ and assume $ \| p \|_{C^{s_0+1}} < 1 / 2 $ with $ s_0 = 1 $,  and
$ \| p \|_{s_0+2} \leq 1 $. 
Then the adjoint operator $ E^* $ in \eqref{eq:E*App} satisfies the tame estimates, 
  for any $ s \geq \max\{1, s_0 + 1 - \sigma \}  $
 \be\label{Estar}
\| E^* u \|_s \leq     
C(s) |h|_{\sigma,s_0+1}   \| u \|_{s+\sigma}  +
C(s) 
\Big( |h|_{\sigma,s+\sigma+s_0}  + |h|_{\sigma, s_0+1}  \| p \|_{s+\sigma+
s_0}  \Big)\| u \|_{s_0+1}  
\ee
where $ C(s) := C(s, \sigma ) > 0 $  are constants  non-decreasing in $ s $. 
\end{lem}

\begin{proof}
In order to estimate 
 $ E^* $ 
  in 
  \eqref{eq:mathcalE*}   we write $ \breve h (x,D) = \widetilde h (x,D) \circ  \langle D \rangle^{\sigma}  $ where $ \widetilde h (x,j) :=   \breve h (x,j)\langle j \rangle^{-\sigma} $
  and thus 
  \[
     \big( \breve h (x,D) \big)^* = 
 \langle D \rangle^{\sigma} \circ \big(  \, \widetilde h (x,D) \big)^*   =  
\langle D \rangle^{\sigma}  \circ  \widetilde h^* (x,D) \, , 
  \]
  where  the symbol
  $ \widetilde h^* (x,j) $   (cfr.~\eqref{adjhD}) is a symbol of order $ 0 $ satisfying, by \eqref{eq:tamekappa0}, \eqref{hsigmah0} and  \eqref{sthbreve}, for any $ s \geq 1 $, 
\begin{equation}\label{passtar}
| \widetilde h^*  |_{0,s} \leq C | \widetilde h|_{0,s+s_0} = 
 C | \breve h|_{\sigma,s+s_0} 
  \leq   C'(s) | h |_{\sigma,s+s_0} + C(s_0) \| p \|_{s+s_0} | h |_{\sigma,s_0+1}  \, . 
\end{equation}
  Thus we estimate
  $ E^*  =  \langle D  \rangle^{\sigma} \circ   \widetilde h^* (x,D) 
  \circ \mathcal P^{-1}\circ \frac{1}{1+p'}  $ as, for any $ s+ \sigma \geq 1 $, 
  setting $ v := {\mathcal P}^{-1}  (\frac{1}{1+p'} u) $
  \begin{align} 
  \| E^* u \|_{s} 
  & \stackrel{\eqref{hsigmah0},\eqref{actiontamepse}} \leq C(s) | \widetilde h^* |_{0,s+\sigma} \| v \|_{s_0} + C(s_0) | \widetilde h^* |_{0,s_0} \| v \|_{s+\sigma} \notag \\
  & \stackrel{\eqref{passtar}} \leq C(s) \Big[  | h |_{\sigma,s+\sigma+s_0} +  \| p \|_{s+\sigma+s_0} | h |_{\sigma,s_0+1}    \Big]  \| v \|_{s_0} + C(s_0)  | h |_{\sigma,s_0+1}      \| v \|_{s+\sigma} \label{intermfioad}
     \end{align}
  since $ s_0 = 1 $ and $ \| p \|_{2 s_0} \leq 1 $.   Now by  
    \eqref{eq:tamecomposition} for $ s = s_0 + 1 $, the condition $ \| p \|_{s_0+2} \leq 1 $,
    \eqref{eq:tameproduct} and the Moser composition 
estimate $ \| \frac{1}{1+p'}\|_s \leq C(s)(1 + \|p \|_{s+1})   $ for any $ s \geq s_0 $, 
 we readily get, for any $ s + \sigma > s_0 + 1 $,  
\be\label{stimavad}
\| v \|_{s+\sigma} \lesssim_s \, \| u \|_{s+\sigma} + \| u \|_{s_0+1} \| p \|_{s+\sigma+1} \, , 
\quad
 \| v \|_{s_0}  \leq \| v \|_{s_0+1} \lesssim_{s_0} \| u \|_{s_0+1} \, . 
\ee
By \eqref{intermfioad} and \eqref{stimavad} we deduce that, 
for any $ s \geq \max\{1, s_0 + 1 - \sigma \}  $,  
\[
  \| E^* u \|_{s} 
  \lesssim_s \Big[  | h |_{\sigma,s+\sigma+s_0} +  \| p \|_{s+\sigma+s_0} | h |_{\sigma,s_0+1}    \Big] \| u \|_{s_0+1}  +   | h |_{\sigma,s_0+1}     \| u \|_{s+\sigma}     
\]
proving \eqref{Estar}.
\end{proof}

{\bf Non-stationary-phase principle.} We shall use the classical 
non-stationary-phase  principle 
in the following form (see \cite[Lemma 13.6]{Baldi:2015}). 

\begin{lem}\label{lem:PS}
Let $ p \in H^2 (\T,\R)  $ satisfy $ \| p \|_2 \leq K $ 
and $ \| p'  \|_{C^0} < 1 / 2 $. Then, for any  $ n, s \in \N $ there exists a constant $C(s,K)>0$ such that 
\be\label{fasesta}
\Big| \int_{\T} u(x) e^{i n (x+ p(x))} \, \diff x \Big| 
%\leq   \frac{\| Q_s\|_0}{n^s}  
\leq  \, \frac{C(s,K)}{n^s}  \Big( \| u \|_s + \| p \|_{s+1}  \| u \|_1 \Big) \, . 
\ee
\end{lem}

\section{The non-linear operator $S$}\label{sec:S}

As a first step, we expand  the non-linear operator $ S(a,b)  $ in \eqref{eq:S} in Fourier series obtaining a formula involving the Bessel function $J_1 (\theta) $ defined in \eqref{Bess1}. To this purpose, note that,
 splitting $e^{-i\varphi}=\cos\varphi-i\sin\varphi$ and using the fundamental theorem of calculus,  the Bessel function $J_1 (\theta) $  in \eqref{Bess1} is equal to 
\begin{equation}\label{J1:prop}
	\begin{aligned}
J_1(\theta)&=\frac{-i}{2\pi} \int_{\T} e^{i \theta \sin \varphi}\sin\varphi\,\diff\varphi \\
& =\frac{-i}{2\pi} \int_{\T} e^{i \theta \sin(\varphi+\pi)}\sin(\varphi+\pi)\,\diff\varphi\\
&=\frac{i}{2\pi} \int_{\T} e^{-i \theta \sin \varphi}\sin\varphi\,\diff\varphi  =\frac{\theta}{2\pi} \int_{\T} e^{-i \theta \sin \varphi}\cos^2\varphi\,\diff\varphi 
	\end{aligned}
\end{equation}
integrating by parts. 
Comparing the first and the third formula we see that $J_1 (\theta) $ is {\it real} and 
{\it odd}.
\begin{lem} \label{lem:SBessel}
{\bf (Fourier expansion of the Action functional)}
The action function $ S(a,b)  $ in \eqref{eq:S}  has the Fourier expansion 
\be\label{SabFourier}
S(a,b) = \sum_{k \in \Z \setminus \{0\}} 
\Big( \frac{1}{k} \int_\T J_1 	\big( k A(x) \big) e^{-i k B(x)} \, \diff x \Big) \,  e_k  
\ee
where $ J_1 $ is the first Bessel function defined in \eqref{Bess1}. 
\end{lem}

\begin{proof}
For any  $k\neq 0$ the Fourier coefficients of the function $ S(a,b) $ in \eqref{eq:S}
are, using the substitution $x=x_{A,B}(I,u)$, \eqref{IAB} and  \eqref{firstintegral},  \begin{align}
	\widehat{S(a,b)}(k) &= \frac{1}{2\pi}\int_\T \int_{\T} (\cos^2\varphi)  \, A(x_{A,B}(I,\varphi)) \, \frac{\partial x_{A,B}}{\partial I}(I,\varphi) \, e^{-i k I} \, \diff \varphi \diff I \nonumber\\
	&=\frac{1}{2\pi}\int_{\T} (\cos^2\varphi) \int_\T A(x) e^{-i k I_{A,B}(x,\varphi)}\, \diff x \, \diff \varphi \nonumber \\
	&= \frac{1}{k}\int_\T  \, e^{-i k B(x)}\frac{kA(x)}{2\pi} \int_{\T} (\cos^2\varphi) e^{- i k A(x) \sin\varphi} \, \diff \varphi
	 \, \diff x \nonumber\\
	&= \frac{1}{k} \int_\T J_1(k A(x)) e^{-i k B(x)} \, \diff x \notag 
\end{align}
by  the last formula in \eqref{J1:prop}. 
% Note that 
% $ \widehat{S(a,b)}(-k) = \overline{\widehat{S(a,b)}(k)} $ since $ J_1 $ is real and odd.
\end{proof}

In order to estimate $ S(a,b) $ and its derivatives we need the following lemma about the decay of  the Bessel function $J_1$ and its derivatives, which follows by elementary properties of oscillatory integrals. 

\begin{lem}\label{lem:J1}
{\bf (Properties of the Bessel function $J_1$)} 
The Bessel function $ J_1 (\theta) $ 
can be written as 
\begin{equation}\label{e:Jr}
J_1(\theta)=  \mathrm{Re} \big(H_1 (\theta) \big) \qquad 
\text{where} \qquad H_1 ( \theta ) = e^{i\theta}r(\theta) 
\end{equation}
has the following properties: 
\begin{enumerate}
\item 
the function $ r\colon\R\to\C$ satisfies 
\begin{equation}
\label{e:r}
|r^{(m)}(\theta)|\leq C_{m} \langle \theta \rangle^{-(m+\frac12)},\qquad \forall\,\theta \in \R  \, ,\ \forall\,m\geq 0\, .
\end{equation}
In particular  the norm 
$ \| r \|_{-\frac12,m} $ defined in \eqref{eq:1lemma1s2} is finite for any $ m \in \N $. 
\item There are $\theta_* > 0 $, $c_* > 0 $ such that 
\be\label{firstHF}
| r(\theta)| \geq c_*|\theta|^{- \frac12} \, , \quad \forall |\theta| > \theta_* \, . 
\ee
\item  For any $ n \in \N $ we have
$ J_1^{(n)} (\theta )  = \mathrm{Re} \big(e^{i \theta } r_n (\theta) \big) $
where 
\begin{equation}
\label{e:rn}
|r_n^{(m)}(\theta)|\leq C_{n, m} 
\langle \theta \rangle^{-(m+\frac12)},\qquad \forall\,\theta \in \R \, ,\ \forall\,m\geq 0\, .
\end{equation}
In particular the norm  $ \| r _n\|_{-\frac12,m} $ defined in \eqref{eq:1lemma1s2} is finite  for any $ m \in \N $.  
\item For any $\epsilon>0$ there exist constants $ 0<\lambda_\epsilon<\Lambda_\epsilon$ such that 
\begin{equation}\label{e:JJ}
\lambda_\epsilon |\theta|^{-1}\leq J_1(\theta)^2+J_1'(\theta)^2
\leq \Lambda_\epsilon|\theta|^{-1}, 
\qquad \text{if }\ \,|\theta|\geq\epsilon\,.
\end{equation}
\end{enumerate}
\end{lem}

\begin{proof} 
Let $\psi_0\colon \T\to[0,1]$ be a smooth function equal to $1$ in a neighborhood of $\varphi=\frac\pi2$ and equal to $0$ in a neighborhood of $\varphi=-\frac\pi2$. Then the function $\psi\colon\T\to[0,1]$ defined by $\psi(\varphi) :=\frac{1}{2}(\psi_0(\varphi)+1-\psi_0(\varphi+\pi))$ has the same property and satisfies $\psi(\varphi)+\psi(\varphi+\pi)=1$ for any $\varphi\in\T$. Therefore we write the Bessel function in \eqref{J1:prop} as
\begin{equation}
\begin{aligned}
J_1(\theta)&=\frac{-i}{2 \pi} \int_{\T} e^{i \theta \sin \varphi}\psi(\varphi)\sin\varphi
\, \diff\varphi+\frac{-i}{2 \pi} \int_{\T} e^{i \theta \sin \varphi}\psi(\varphi+\pi)\sin\varphi \, \diff\varphi\\
&=\frac{-i}{2 \pi} \int_{\T} e^{i \theta \sin \varphi}\psi(\varphi)\sin\varphi
\, \diff\varphi+\frac{i}{2 \pi} \int_{\T} e^{-i \theta \sin \varphi}\psi(\varphi)
\sin\varphi \, \diff\varphi =
\mathrm{Re} \big(H_1 (\theta) \big)
\end{aligned}	
\end{equation}
as in \eqref{e:Jr}
with
\[
H_1(\theta):=\frac{-i}{ \pi} \int_{\T} e^{i \theta \sin \varphi}\psi(\varphi)\sin\varphi \, \diff\varphi \, ,
\qquad r(\theta):=\frac{-i}{ \pi} \int_{\T} e^{i \theta (\sin \varphi-1)}\psi(\varphi)\sin\varphi \, \diff\varphi \, . 
\]
Note that  $ r(\theta)=-\overline{r(-\theta)}$. 
The function $r(\theta)$ is an oscillatory integral whose 
phase $\phi(\varphi):=\sin\varphi-1$ has a unique critical point in the support of $\psi$ at $\varphi=\frac\pi2$. Since $\phi(\tfrac\pi2)=0$, the decay bounds \eqref{e:r} follow from \cite[Chapter VIII, Prop. 1, Prop. 3]{Stein}. We now prove the lower bound  \eqref{firstHF}.  
By \cite[Chapter VIII, Prop. 3]{Stein}   the 
oscillatory integral $ r(\theta ) $
admits the asymptotic expansion for $\theta>0$ large
\begin{equation}\label{rtlbo}
r (\theta) = \frac{-i}{ \pi} \Big( \frac{a_0(1)}{\sqrt{\theta}} + r_{1} (\theta) \Big) \qquad 
\text{where} \qquad 
|r_1 (\theta)| \leq \frac{C}{\theta} 
\end{equation}
and 
by \cite[Chapter VIII, Remark 1.3.4]{Stein}
\[
a_0 (1) = \Big( \frac{2\pi}{ - i \phi^{''} ( \tfrac{\pi}{2})} \Big)^{1/2} \psi(\tfrac\pi2)\sin(\tfrac\pi2)
=\Big( \frac{2\pi}{i} \Big)^{1/2}. 
\]
Since $a_0 (1) $ is  different from zero,  
the lower bound \eqref{firstHF} follows by \eqref{rtlbo}.

The estimates \eqref{e:rn} follow by induction from \eqref{e:r}. 

Finally, 
by \eqref{e:Jr},  \eqref{e:r}, \eqref{firstHF} we deduce that, for any $ | \theta | > \theta_*  $, 
\begin{align}
J_1(\theta)^2+J_1'(\theta)^2
& =\Big(\mathrm{Re}(e^{i\theta}r(\theta))\Big)^2+\Big(\mathrm{Re}(e^{i\theta}(ir(\theta)+r'(\theta))\Big)^2 
\notag \\
	&=\Big(\mathrm{Re}(e^{i\theta}r(\theta))\Big)^2+\Big(\mathrm{Im}(e^{i\theta}r(\theta))\Big)^2+O(|\theta|^{-2}) \notag \\
	&=|r(\theta)|^2+O(|\theta|^{-2}) = c^2_* |\theta|^{-1}+O(|\theta|^{-2}) \label{boundJ1J1'}
\end{align}
which proves \eqref{e:JJ} for $ |\theta | \geq \theta_1 $ large enough. 
In order to prove  \eqref{e:JJ} for $ \epsilon \leq |\theta | \leq \theta_1 $,
it is sufficient to show that 
the functions $J_1 (\theta) $ and $J_1' (\theta)  $ do not vanish simultaneously.
Assume by contradiction that there exists  $ \theta_0 \neq 0 $ such that 
$ J_1(\theta_0)=J_1'(\theta_0)=0$.
Since $J_1 (\theta) $ solves the second-order linear differential equation 
\[
\theta^2 J_1'' (\theta)+\theta J_1'(\theta)+(\theta^2-1)J_1(\theta)=0 
\] 
this would imply that  $J_1(\theta)=0$ for all $\theta\in\R $, which is false, see e.g. \eqref{boundJ1J1'}. Thus \eqref{e:JJ} is proved. 
\end{proof}

\begin{lem}\label{lem:tameS}
{\bf (Tame estimates for $ S $)}
	 There is $ \delta > 0 $ such that if $ \| (a,b) \|_{3} < \delta $, 
	then, for any $s\geq \tfrac72 $ we have 
	\begin{equation}\label{eq:tameS}
	\|S (a,b)\|_{s} \leq C(s) \big( 1+ \|(a,b)\|_{s- \frac12} \big) \, .
	\end{equation}
\end{lem}

\begin{proof}
We use \eqref{AxBx}, \eqref{SabFourier} and \eqref{e:Jr}  
to write 
\begin{align}
& S(a,b) 
		 = \sum_{k\neq 0} \left(\frac{1}{k} \int_\T \mathrm{Re} \big( r(k A(x)) e^{ i k A(x)}\big ) e^{- i k B(x)} \, \diff x\right) e_k  \label{eq:FourierS2} \\
		& = \sum_{k\neq 0} \frac{1}{2k} \left ( \int_\T r( k A(x))  
		e^{i k A_*} e^{- i k (x+ b(x)-a(x))} \, \diff x\right )e_k   + \sum_{k\neq 0} \frac{1}{2k} \left ( \int_\T 
		\overline{r( k A(x))}  e^{- i k A_*} 
		e^{- i k (x+ b(x) +a(x))} \, \diff x\right) e_k  \, .  \notag 
\end{align}
The first series in \eqref{eq:FourierS2} (being the argument for the second series completely analogous) has  the form $ E^*$
as in \eqref{eq:E*App} with $u(x)\equiv 1$ and 
\[
h(x,k) := \begin{cases} 
(2k)^{-1} \overline{r( k A(x))} e^{-i k A_*}\, , \quad \ \ \forall k \neq 0, \\
0  \, , \qquad \qquad \qquad \qquad \qquad \quad  k = 0 \, , 
	\end{cases} \qquad \qquad p(x) := b(x) - a(x) \, . 
\] 
In view of 
\eqref{e:r} and  Lemma \ref{lem:Sobosimbol}, the symbol $ h(x,k) $ has order 
$ -3/2 $ with  norm  
$ | h |_{-\frac32,s} \lesssim_s 1+ \| a\|_s  $ for any $ s \geq 0 $. Indeed, notice that up to the factor $(2k)^{-1}e^{-ikA_*}$ which satisfies $|(2k)^{-1}e^{-ikA_*}| \leq \langle k\rangle^{-1}$, $h$ is of the form \eqref{eq:symbolh}, where $\ell$ has order $\sigma=-\tfrac12$.
 Then applying Lemma \ref{lem:Estargene} with $ u = 1 $ and $ \sigma = - \frac32 $,
 and for $ \| (a,b) \|_{s_0+2}=\| (a,b) \|_{3} $ small enough, 
 we deduce,  for any $ s \geq \max\{1,s_0 + 1 - \sigma \} =\max\{1,1+1+\tfrac32\}= \tfrac72 $,  the bound in \eqref{eq:tameS}.
\end{proof}

%%%%%%%%%%%%%%%
%%%%%%%%%%%%%%%
%%%%%%%%%%%%%%%

\section{The linearized operator $ \diff S $} \label{sec:diffS}
In this section we compute the differential $\diff S\colon C^\infty(\T)\times C^\infty(\T)\to C^\infty_0(\T)$, the second differential $ \diff^2 S $, the $L^2$-adjoint $\diff S^*\colon C^\infty_0(\T)\to C^\infty(\T)\times C^\infty(\T)$, and the composition $\diff S\circ \diff S^*\colon C^\infty_0(\T)\to C^\infty_0(\T)$. All these operators are real, 
namely send real functions into real functions.
We also prove that such operators 
satisfy 
tame estimates in Sobolev spaces.

\noindent 
{\bf The differential $\diff S $.}
Differentiating \eqref{SabFourier}
at $(a,b)$ in the tangent direction $(\alpha,\beta)$, 
we deduce that
\be\label{diffSab}
\diff S(a,b)[\alpha,\beta] = 
 \sum_{k\neq 0} \left ( \int_\T \Big [ J_1'( k A(x))  \alpha(x) - i J_1 ( k A(x))  \beta(x) \Big ]  e^{- i k B(x)}\, \diff x\right ) e_k \, . 
\ee
We first consider the case $ (a,b) = (0,0) $.

\begin{lem}\label{lem:dStrivial}
{\bf (Differential $ \diff S(0,0) $)}
The differential
\be\label{dSacon}
\diff S(0,0)[\alpha,\beta] =  2 \pi
\sum_{k\neq 0} \Big [J_1'( k A_*)\widehat\alpha(k)-i J_1 ( k A_*) \widehat\beta(k) \Big ] e_k 
\ee
is a map 
$ \diff S(0,0) \colon H^{s-\frac 12}(\T)\times H^{s-\frac 12}(\T)\to H^{s}_0(\T) $  for any  $ s \in \R $, with  kernel 
\[
\ker \diff S(0,0) = \Big \{(\alpha,\beta) 
\in H^{s-\frac 12} (\T)\times H^{s-\frac 12}(\T) \ \Big |\ J_1'(k A_* ) \widehat \alpha (k) - i J_1(k
A_*) \widehat \beta(k)=0 \, , \ \forall k\neq 0\Big \} 
\]
(cfr.~\eqref{SDABstar}) and  
right inverse
$ R(0,0)  \colon H^{s+\frac12}_0(\T) \to H^s(\T)\times H^s(\T)  $ given by 
\[
R(0,0)  [\gamma] = 
\frac{1}{2 \pi}  
\Big (  \sum_{j\neq0} \frac{J_1'( jA_*)}{J_1( jA_*)^2 + J_1'(jA_*)^2} \widehat \gamma(j)e_j, i \sum_{j\neq 0} 
\frac{J_1(j A_*)}{J_1( jA_*)^2 + J_1'(jA_*)^2}  \widehat\gamma(j)e_j\Big ) \, . 	 
\]
\end{lem}

\begin{proof}
Formula \eqref{dSacon} is a special case of \eqref{diffSab} with  
$ \widehat \alpha (k) = \frac{1}{2\pi} \int_{\T} \alpha (x) e^{-i k x } \,  \diff x $
and $ \widehat \beta (k) = \frac{1}{2\pi} \int_{\T} \beta (x) e^{-i k x } \,  \diff x $. 
By \eqref{e:r}, \eqref{e:rn}
we deduce that $ |J_1 (k A_*) |,  |J_1' (k A_*) | 
\lesssim \langle k \rangle^{-\frac12} $ and, using also \eqref{e:JJ}, the lemma follows. 
\end{proof}

The differential $ \diff S(a,b)$ satisfies the following tame  estimates on Sobolev spaces.  

\begin{lem} {\bf (Tame estimates of $ \mathrm d S $)}	\label{lem:tamedS}
There is $ \delta > 0 $ such that if $ \| (a,b) \|_{ 3} < \delta $,  
	then, for any $s\geq \tfrac52 $ we have  
	\begin{equation}\label{eq:tamedS}
	\|\diff S (a,b)[\alpha,\beta]\|_{s} \leq C(s) \|(\alpha,\beta)\|_{s- \frac12} 
	+ C(s) \|(a,b)\|_{s+ \frac12} \|(\alpha,\beta)\|_2 \, .
	\end{equation}
\end{lem}
\begin{proof}
From \eqref{diffSab} we deduce that $\diff S(a,b)$ is the sum of four operators of the form $E^*$ in \eqref{eq:E*App}. More precisely, using the notation of \eqref{eq:E*App} the four operators correspond to
\begin{align*}
u(x)=\beta (x) \, ,\quad h(x,k)=\frac i2e^{-ikA_*} \overline{r(kA(x))}, \forall k \neq 0 \, , 
h(x,0) = 0 \, , 
\quad p(x)=b(x)- a(x) \, ,\\
u(x)=\beta (x) \, ,\quad h(x,k)=\frac i2e^{ikA_*}r(kA(x)), \forall k \neq 0 \, , 
h(x,0) = 0 \, ,
\quad p(x)=b(x)+a(x) \, ,\\
u(x)=\alpha(x) \, ,\quad h(x,k)=\frac 12e^{-ikA_*}\big(-i\overline{ r(kA(x))}+ \overline{r'(kA(x))}\big),
\forall k \neq 0 \, , 
h(x,0) = 0 \, , \quad p(x)=b(x) -a(x) \, ,\\
u(x)=\alpha(x) \, ,\quad h(x,k)=\frac 12e^{ikA_*}\big(ir(kA(x))+r'(kA(x))\big),
\forall k \neq 0 \, , 
h(x,0) = 0 \, , \quad p(x)=b(x) + a(x) \, .
\end{align*}
%{\color{green} VERSIONE GABRIELE
%\begin{align*}
%u(x)=\alpha(x),\quad h(x,k)=\frac i2e^{ikA_*} \overline{r(kA(x))},\quad p(x)=a(x)+b(x),\\
%u(x)=\alpha(x),\quad h(x,k)=\frac i2e^{-ikA_*}r(kA(x)),\quad p(x)=a(x)-b(x),\\
%u(x)=\beta(x),\quad h(x,k)=\frac 12e^{ikA_*}\big(-i\overline{ r(kA(x))}+ \overline{r'(kA(x))}\big),\quad p(x)=a(x)+b(x),\\
%u(x)=\beta(x),\quad h(x,k)=\frac 12e^{-ikA_*}\big(ir(kA(x))+r'(kA(x))\big),\quad p(x)=a(x)-b(x).
%\end{align*}}
Up to the factor $e^{\pm ikA_*}$ which is constant in $x$ and has absolute value $1$ for all $k$, the symbols $h$ are of the form \eqref{eq:symbolh}, where the corresponding function $\ell$ is of order $\sigma=-1/2$ by \eqref{e:r}. As a consequence, Lemma \ref{lem:Sobosimbol} implies that $ | h |_{-\frac12,s} \lesssim_s 1 + \| a \|_s $ for any $ s \geq 0 $, 
cfr.~Lemma \ref{lem:tameS}. 
Therefore Lemma \ref{lem:Estargene} with $ p(x) = b(x) \pm a(x) $ (which satisfies 
$ \| p \|_{C^{s_0+1}}, \| p \|_{s_0+2} \leq 1 $)
implies, for any $ s   \geq \max\{1,s_0 + 1 - \sigma \} = \tfrac52 $,  the tame estimates \eqref{eq:tamedS} for $\diff S$. 
\end{proof}

\noindent 
{\bf The differential $\diff^2 S$.}
Differentiating \eqref{diffSab}
with respect to $(a,b)$ 
in the direction $(\alpha_1,\beta_1)$, we obtain
\begin{align}
 \diff^2 S(a,b)[(\alpha,\beta),(\alpha_1,\beta_1)]   
& =   \sum_{k \neq 0} \Big( \int_\T k \Big [  J_1''(k A(x))  \alpha (x)\alpha_1 (x) - J_1( kA(x))  \beta(x) \beta_1 (x)  \label{eq:coefficientsdS^2} \\
	&  \quad \qquad \qquad \qquad - i J_1' ( k A(x))  \big (\alpha_1(x)\beta(x) + \alpha(x)\beta_1(x)\big ) \Big ]  e^{-i k B(x)}  \, \diff x \Big) e_k \, . \nonumber 
\end{align}
The following tame estimates for $ \diff^2 S $ hold. 
\begin{lem} {\bf (Tame estimates of $ \diff^2 S $)}	\label{lem:tamedS2}
There is $ \delta > 0 $ such that if $ \| (a,b) \|_{3} < \delta $,  
	then, for any $ s \geq 3/2 $,  we have  
	\begin{align}
	\|\diff^2 S & (a,b) [(\alpha,\beta),(\alpha_1,\beta_1)]\|_{s} \leq 
	\label{eq:tamedS^2} \\
	&  C(s) \|(\alpha,\beta)\|_{2}\|(\alpha_1,\beta_1)\|_{s+\frac12} + 
	C(s) \|(\alpha,\beta )\|_{s+\frac12} \|(\alpha_1,\beta_1)\|_{2} 
	 + C(s) \|(a,b)\|_{s+ \frac32} 
	\|(\alpha,\beta)\|_{2}\|(\alpha_1,\beta_1)\|_{2} \nonumber 
	\end{align}
	and, in addition
\be\label{dSa-dS0small}
\Big\| \Big( \diff S (a,b)-\diff S (0,0) \Big) [\alpha,\beta] \Big\|_s 
\leq C(s) \| (a,b) \|_{2} \| (\alpha,\beta) \|_{s+\frac12} 
+ C(s) \| (a,b) \|_{s+ \frac32} \| (\alpha,\beta) \|_{2} \, . 
\ee
\end{lem}

\begin{proof}
From \eqref{eq:coefficientsdS^2} we deduce that 
$\diff^2 S(a,b)$ is sum of operators of the form $E^*$ in \eqref{eq:E*App} with 
$u $ one of the functions $ \alpha \alpha_1,\beta\beta_1,\alpha\beta_1,\alpha_1\beta$ 
and symbols $h$ as in \eqref{eq:symbolh} of order $\sigma=1/2$  by Lemmata \ref{lem:J1} and \ref{lem:Sobosimbol}, satisfying 
$ | h |_{\frac12,s} \lesssim_s 1 + \| a \|_s $ for any $ s \geq 0 $, 
cfr.~Lemma \ref{lem:tameS}. Therefore Lemma \ref{lem:Estargene} 
with $ p(x) = b(x) \pm a(x) $ (which satisfies 
$ \| p \|_{C^{s_0+1}}, \| p \|_{s_0+2} \leq 1 $)  implies,
for any $ s   \geq \max\{1,s_0 + 1 - \sigma \} = 3/ 2  $,  the tame estimates \eqref{eq:tamedS^2}. 

The bound \eqref{dSa-dS0small}
follows by writing  
$ \diff S (a,b)-\diff S (0,0) = \int_0^1 \diff^2 S ( \tau a, \tau b) [a,b] \, \diff \tau $  
and using \eqref{eq:tamedS^2}.
\end{proof}

\noindent {\bf Adjoint operator $\diff S^*$. }
We now write the adjoint of the operator 
$\diff  S (a,b)  $ defined in \eqref{diffSab}
with respect to the $L^2$-scalar product.

\begin{lem} {\bf (The adjoint $\diff S^*$ and $ (\diff^2 S)^*$)} \label{lem:adjo}
For any  $\gamma \in C^\infty_0(\T)$ it results
\begin{equation}\label{eq:dS*}
\diff  S (a,b)^* [\gamma] = 2 \pi \Big ( \sum_{j\neq0} J_1'( jA(x)) e^{ij B(x)} \widehat \gamma(j), i \sum_{j\neq 0} J_1(jA(x)) e^{ i j B(x)} \widehat\gamma(j)\Big ).
\end{equation}
Moreover
\begin{equation}\label{eq:d2S*}
\big( \diff^2  S (a,b)[\alpha,\beta] \big)^* [\gamma] = 
2 \pi 
\begin{pmatrix} 
 \sum_{j\neq0} j 
\big( J_1'' ( jA(x)) \alpha (x) + i J_1'(j A(x)) \beta (x) \big) 
e^{ij B(x)} \widehat \gamma(j) \\
 \sum_{j\neq 0} j 
\big(   i J_1' ( jA(x)) \alpha (x) -J_1(j A(x)) \beta (x) \big)  e^{ i j B(x)} \widehat\gamma(j)  
\end{pmatrix} \, .
\end{equation}
%{\color{green} GABRIELE VERSION: 
%$$ % \begin{equation}\label{eq:d2S*}
%\big( \diff^2  S (a,b)[\alpha,\beta] \big)^* [\gamma] = 
%2 \pi 
%\begin{pmatrix} 
% \sum_{j\neq0} j 
%\big( J_1'' ( jA(x)) \alpha (x) + i J_1'(j A(x)) \beta (x) \big) 
%e^{ij B(x)} \widehat \gamma(j) \\
% \sum_{j\neq 0} j 
%\big( - iJ_1' ( jA(x)) \alpha (x) -J_1(j A(x)) \beta (x) \big)  e^{ i j B(x)} \widehat\gamma(j)  
%\end{pmatrix}
%$$}% \end{equation}
\end{lem}

\begin{proof}
By \eqref{diffSab}, for any $ j \in \Z $,    
 \begin{align*}
\big(  \diff S(a,b)[\alpha,\beta], e_j \big)_{L^2} 
=  \Big( (\alpha,\beta ), \underbrace{2 \pi  \Big ( J_1'( jA(x)) e^{ ij B(x)}, i J_1( j
A(x)) e^{ i j B(x)}\Big )}_{= \diff S(a,b)^* [e_j]} \Big)_{L^2 \times L^2}
\end{align*}
by  definition of adjoint, proving \eqref{eq:dS*}. 
Similarly \eqref{eq:d2S*} follows by \eqref{eq:coefficientsdS^2}.
\end{proof}

 The following tame estimates for the adjoint operator $\diff S^*$ hold. 

\begin{lem}
{\bf (Tame estimates for $ (\diff S)^* $ and $ (\diff^2 S)^* $)}\label{dSadj}
There is $ \delta > 0 $ such that if $ \| (a,b) \|_{3} < \delta $,  
	then, 
for any $s\geq 2$,  we have, for any $ \gamma \in H^{s- \frac12}_0(\T) $,  
\begin{equation}\label{eq:tamedS*}
\|\diff S(a,b)^*[\gamma]\|_s \leq C(s) \|\gamma\|_{s- \frac12} + 
C(s) \|(a,b)\|_s \|\gamma\|_{\frac32} \, .
\end{equation}
Moreover, for any $s\geq 2$, 
\be\label{d2Sstar}
\begin{aligned}
&	\Big\| \Big( \diff^2 S  (a,b) (\alpha,\beta) \Big)^* [\gamma] \Big\|_{s} \leq
	\\
& C(s) \|(\alpha,\beta)\|_{2}\| \gamma \|_{s+\frac12}  +
 C(s)	\|(\alpha,\beta)\|_{s} \| \gamma \|_{\frac52}
+ C(s) \|(a,b)\|_{s} 
\|(\alpha,\beta)\|_{2}\| \gamma \|_{\frac52} 
	\end{aligned}
\ee
and
	\be\label{dS0d2}
\Big\| \Big( \diff S^* (a,b)-\diff S^* (0,0) \Big) \gamma  \Big\|_s 
\leq C(s) \| (a,b) \|_{2} \| \gamma \|_{s+\frac12} 
+ C(s) \| (a,b) \|_{s} \| \gamma \|_{\frac52} \, . 
	\ee
\end{lem}

\begin{proof}
Each component of the operator 
$\diff S(a,b)^*$ in \eqref{eq:dS*} is the sum of FIO 
of the form $E$ as in \eqref{eq:EApp}, with  symbols of order $-1/2$, by  Lemma \ref{lem:J1}. 
For instance, recalling that $ J_1 (\theta ) = \mathrm {Re} ( e^{i \theta} r(\theta) ) $, 
the second component is equal to
\[
 \pi i \left ( \sum_{j\neq 0} r( jA(x)) e^{ i j A_*}
e^{ i j (x + b(x) + a(x))} \widehat\gamma(j) -  \sum_{j\neq 0} 
\overline{r(jA(x))} e^{- i jA_* } e^{i j (x+b(x) -a(x))} 
\widehat\gamma(j)\right )
\]
and \eqref{eq:tamedS*} follows from Lemma \ref{lem:Eusmall}  with $ u \equiv \gamma $,
for any $ s \geq s_0 + 1 = 2 $, 
with a symbol satisfying 
$ | h |_{-\frac12,s} \lesssim_s 1 + \| a \|_s $ for any $ s \geq 0 $, by using Lemmata
\ref{lem:J1} and \ref{lem:Sobosimbol}, cfr.~the proof of Lemma \ref{lem:tameS}. 
 The estimates \eqref{d2Sstar} follow similarly by \eqref{eq:d2S*} and 
Lemmata \ref{lem:J1}, \ref{lem:Sobosimbol} and \ref{lem:Eusmall}   with 
symbols of order $1/ 2$ satisfying 
$ | h |_{\frac12,s} \lesssim_s \| (\alpha, \beta)  \|_s + \| a \|_s \| (\alpha, \beta) \|_{s_0}$ 
for any $ s \geq s_0 $.
\end{proof}

We also provide the matrix representation of the composite operator 
$ \diff  S (a,b) \circ \diff  S(a,b)^* $. 
\begin{lem}\label{matrixentries}
For any $ \gamma \in C^\infty_0 (\T) $ it results  
\begin{align}\label{eq:dSdS*}
& \diff  S (a,b) \circ \diff  S(a,b)^*[\gamma] \\
			&= 2 \pi\,  \sum_{j,k\neq 0} \left ( \int_\T 
			\Big ( J_1(kA(x))J_1(jA(x)) + 
			J_1'( kA(x))J_1'( jA(x))\Big ) 
			e^{ i (j-k) B(x)} \, \diff x \right ) 
			\widehat\gamma(j) e_k \, . \nonumber 
\end{align}
For $ (a,b) = (0,0) $ it reduces to the diagonal operator
 \be\label{dSdS0}
\diff S(0,0)\circ\diff S(0,0)^*[\gamma]
= 4 \pi^2 \,  \sum_{j\neq 0} \Big ( J_1( jA_*)^2 
+ J_1'( jA_*)^2\Big )\widehat\gamma(j) e_j \, . 
\ee
In particular, $ \diff S(0,0)\circ\diff S(0,0)^* \colon H^{s-1}_0(\T)\to H^{s}_0(\T) $ is a bounded invertible operator for any $s\in \R$.
\end{lem}

\begin{proof}
By \eqref{eq:dS*} and \eqref{diffSab} 
we readily deduce \eqref{eq:dSdS*}.  The last statement follows by  \eqref{e:JJ}. 
\end{proof}

In order to apply the Nash--Moser implicit function theorem, a key step is to 
prove the existence of a right inverse for $\diff S (a,b) $ in a neighborhood of $(0,0)$ in
a low norm. 
This will be done in the next section showing that the composition 
$\diff S(a,b)\circ \diff S^* (a,b)$ is invertible in a neighborhood of $(0,0)$ and 
 taking the right inverse  
\be\label{Rab}
R(a,b) := \diff S(a,b)^* \circ (\diff S (a,b)\circ \diff S(a,b)^*)^{-1}  
\ee
of $\diff S (a,b) $.

\section{Invertibility of $ \diff S (a,b) \circ \diff S^* (a,b) $ and Proof of
Theorem \ref{thm:ri}}\label{dSdSstar}

In this section we prove Theorem \ref{thm:ri}. 
The next key Proposition \ref{prop:key} 
proves the invertibility 
of the  real self-adjoint operator 
\be\label{opAab}
M := M(a,b) := \diff S (a,b) \circ \diff S^* (a,b)
\ee
and that its inverse satisfies  tame estimates. 

\begin{prop}\label{prop:key} 
{\bf (Inverse of $ \diff S (a,b) \circ  \diff S^* (a,b)$)}
There exists $ \delta > 0 $ such that, if 
\be\label{eq:alphabsmall}
\| (a,b) \|_{6} < \delta
\ee
then
the operator  $ M(a,b) $ in \eqref{opAab}
 is invertible and, for any  $ s \geq 2 $
and any $ \gamma \in H^{s+1 }_0(\T) $, 
\begin{equation}\label{Mab-1}
\| M(a,b)^{-1}  \gamma \|_s 
 \leq C (s) \|  \gamma \|_{s+1} +  C(s) \| (a, b) \|_{s+4} \|  \gamma  \|_{3}    \, . 
\end{equation}
\end{prop}

\noindent 
{\sc Proof of Theorem  \ref{thm:ri}.}
In view of  Proposition \ref{prop:key} for any $ \|(a, b)\|_{6} < \delta $, 
the right inverse $ R(a,b) $ in \eqref{Rab} is well defined 
and  
the tame estimates \eqref{Rabfinal} of the right-inverse operator $ R(a,b) $  defined in   
\eqref{Rab} follow by \eqref{eq:tamedS*} 
and \eqref{Mab-1}. $ \hfill \qed $

\medskip

We now start the proof of Proposition \ref{prop:key}. 

\paragraph{\bf The resolvent identity decomposition.}
Given $ N \in \N $, fixed in Proposition \ref{invpropAtilde} below, 
we apply a  resolvent identity procedure. Remind that $ M $ acts between spaces 
of $ 2 \pi $-periodic functions with  zero average. 
According to the orthogonal 
splitting 
\be\label{HHLHR}
\begin{aligned}
 H^s_0 (\T) = H_L \oplus H_R  \, , \quad
& H_L := \Big\{ u = \sum_{0<|j| \leq N} u_j e_j  \Big\} \, , \quad 
H_R := \Big\{ u = \sum_{|j| > N} u_j e_j  \in H^s_0 (\T)  \Big\} \, , 
\end{aligned}
\ee
the operator $ M $ is represented by the matrix of operators 
\be\label{Adeco}
M \equiv \begin{pmatrix}
M^L_L & M^R_L  \\
M^L_R & M^R_R
\end{pmatrix} 
\ee
where $ M_L^L := \Pi_L M \Pi_L  $,  $ M_R^L := \Pi_R M \Pi_L $ and 
$ M_R^R := \Pi_R M \Pi_R $,  $ M_L^R := \Pi_L M \Pi_R $ where 
$ \Pi_L $, $ \Pi_R $
denote respectively the $ L^2$-projectors on $ H_L $ and $ H_R $,  
\[ 
\Pi_L u = \sum_{0 < |j|\leq N} (u, e_j)_{L^2} e_j\, , \quad \Pi_R  u =  \sum_{|j| > N} (u, e_j)_{L^2} e_j \, . 
\]
The following simple algebraic lemma holds. 

\begin{lem}\label{lem:resol}
{\bf (Resolvent identity)}
Assume that $ M^L_L $ is invertible. Then $ M $ is invertible if and only if  
\begin{equation}\label{ARRsmall}
\tilde M_R^R :=  M_R^R -  M_R^L (M^L_L)^{-1} M^R_L
\end{equation}
is invertible and
\be\label{resoid}
M^{-1} = 
\begin{pmatrix}
(M^L_L)^{-1} + (M^L_L)^{-1} M^R_L (\tilde M_R^R)^{-1} M_R^L (M^L_L)^{-1}
& \quad - (M^L_L)^{-1} M^R_L (\tilde M_R^R)^{-1}   \\
 - (\tilde M_R^R)^{-1} M^L_R (M^L_L)^{-1}  & \quad (\tilde M_R^R)^{-1} 
\end{pmatrix} \, . 
\ee
\end{lem}

As first step we  prove 
the invertibility of $ M $
restricted to low modes. 
\begin{lem}\label{invALL}
{\bf (Invertibility of $ M_L^L $)}
For any $ N \in \N $, 
if \eqref{eq:alphabsmall} holds with  $  \delta := \delta (N)$ small enough,
\[ 
M_L^L := M_L^L (a,b) = \Pi_L \diff S (a,b)  \diff S^* (a,b) \Pi_L 
\] 
is invertible and  $ \| (M_L^L)^{-1} \| \lesssim_N 1 $. 
\end{lem}

\begin{proof}
The operator $ M_L^L (a,b)$ is  the 
finite dimensional  operator 
 obtained restricting in \eqref{eq:dSdS*}  the  indexes 
$ 0 < |j|, |k | \leq N $. The operator $ M_L^L (0,0)$ 
is invertible by 
\eqref{dSdS0} and \eqref{e:JJ}, and
thus, since  \eqref{eq:alphabsmall} holds with  $  \delta := \delta (N)$
 small enough,  
$ M_L^L (a,b) $ is invertible  and the lemma is proved. 
\end{proof}

We also estimate the coupling operators $ M_L^R $ and $ M^L_R  $. 

\begin{lem} \label{lem:coupling}
{\bf (Coupling operators)}
The real
operators $ M_L^R $ and $ M^L_R    $  have the finite rank form 
\be\label{Aalpha0}
M_L^R = 
\sum_{0<|j| \leq N} ( \cdot, g_j  )_{L^2} e_j \, , 
\quad 
M^L_R = 
\sum_{0<|j| \leq N} ( \cdot, e_j  )_{L^2} g_j
\ee 
where each $ g_j $ is a  smooth function in $H_R$
satisfying, for any $ s \geq \tfrac52 $, 
\be\label{estgj}
\sup_{0<|j| \leq N}\| g_j \|_{s} \leq C(s,N) \| (a, b)\|_{s+ \frac32} \, . 
\ee 
In particular  
\be\label{MLRtame}
\begin{aligned}
& \| M_L^R  \gamma \|_s   \leq C(s,N)   \| (a, b)\|_4 \| \gamma  \|_0 \, , \qquad 
\forall s \geq 0 \, , 
\\
& \| M^L_R  \gamma  \|_s   \leq C(s,N)   \| (a, b)\|_{s+\frac32} \| \gamma  \|_0 
\, , \quad 
\forall s \geq \tfrac52  \, . 
\end{aligned} 
\ee
\end{lem}

\begin{proof}
Since the operator $ M(0,0) $ 
is diagonal (Lemma \ref{dSdS0})
we have that $ \Pi_L  M (0,0)  \Pi_R = 0 $ and,  
setting $ \Delta  %\Delta (a, b) 
:= M(a,b) - M(0,0) $,  we get  
\[
M_L^R \gamma =  \Pi_L \Delta  \Pi_R \gamma = \sum_{0<|j| \leq N} (  \Delta 
\Pi_R \gamma , e_j)_{L^2} e_j 
 =
\sum_{0<|j| \leq N}(   \gamma , \Pi_R \Delta  e_j)_{L^2} e_j 
\]
being $ \Delta $  self-adjoint as $ M $ in
\eqref{opAab}.
This proves that $ M_L^R $ has the form \eqref{Aalpha0}   with $ 
g_j := \Pi_R \Delta  e_j $. The form of $  M^L_R  = (M_L^R)^*  $ 
follows by  \eqref{RFRadj}. 
Recalling  \eqref{opAab}, we have
\[
\Delta = M(a,b) - M(0,0) = \Big(\diff S (a,b)-\diff S (0,0) \Big) \diff S (a,b)^*+
\diff S (0,0) \Big(  \diff S (a,b)-\diff S (0,0)  \Big)^*  \, .
\]
Then, using \eqref{dSa-dS0small}, \eqref{eq:tamedS*}, Lemma \ref{lem:dStrivial} 
%\eqref{eq:dS*}, \eqref{eq:tamedS} 
and \eqref{dS0d2},  
 we deduce that each $ g_j  =  \Pi_R \Delta  e_j $ 
 satisfies \eqref{estgj}. 
 
Finally \eqref{MLRtame} follows by \eqref{Rure} and \eqref{Aalpha0}-\eqref{estgj}.
\end{proof}

In view of Lemmata \ref{invALL} and  \ref{lem:resol}, 
in order to prove the invertibility of $ M $ and the tame estimates for 
$ M^{-1}$, it is sufficient to prove that the real operator 
$ \tilde M_R^R =  M_R^R -  M_R^L (M^L_L)^{-1} M^R_L  $
in \eqref{ARRsmall} is invertible and its inverse satisfies tame estimates. 
The next proposition  is the key result.

\begin{prop}\label{lem}\label{invpropAtilde}
{\bf (Inverse of $\tilde M_R^R $)}
There is  $ N \in \N $ such that, if \eqref{eq:alphabsmall} 
holds with $ \delta $ small enough, then
the operator $ \tilde M_R^R $ in \eqref{ARRsmall} is invertible and, for any  $ s \geq 2 $, 
  \[
\| (\tilde M_R^R)^{-1}  \gamma \|_s 
 \leq C (s) \| \gamma \|_{s+1} +  C(s) \| (a, b) \|_{s+4} \|  \gamma  \|_{3}    \, . 
\]
\end{prop}

\noindent 

Before proving Proposition \ref{invpropAtilde} we show how 
it implies Proposition \ref{prop:key}. 

\noindent 
{\sc Proof of Proposition \ref{prop:key}.}
 It is a direct consequence of 
Proposition \ref{invpropAtilde}, Lemmata \ref{lem:resol},  \ref{invALL}, and 
\ref{lem:coupling}. 
%We also use 
%that the composition of tame operators is
%a tame operator with tame constants as in \eqref{CM1M2}. $ \qed $
 
 \medskip
 
\paragraph{\bf Proof of Proposition \ref{invpropAtilde}.}
The derivative of  the Bessel function 
$ J_1(\theta) = \mathrm{Re} (H_1(\theta)) $, 
with $H_1(\theta) = r(\theta)e^{i\theta} $ 
as  in Lemma \ref{lem:J1}, is   
\be\label{defK}
J_1'(\theta) = - \mathrm{Im} (H_1(\theta)) + K(\theta) \qquad \text{with} \qquad K(\theta) := \mathrm{Re} (r'(\theta) e^{i\theta}) \, . 
\ee
The fact that  $ r (\theta ) \sim \langle \theta \rangle^{- \frac12 } $  
whereas
$ r' (\theta ) 	\sim  \langle \theta \rangle^{- \frac32 } $ 
 for $ |\theta| \to + \infty $ (Lemma \ref{lem:J1}), suggests 
to decompose the linear operator $ \diff S(a,b) $ in \eqref{diffSab} as
\be\label{SabE0E1}
\diff S(a,b) = E_0(a,b) + E_1(a,b)
\ee
where
\begin{align}
E_0(a,b)[\alpha,\beta] &:=\sum_{k\neq 0} \Big ( -    \int_\T \Big 
[ \mathrm{Im} (H_1( kA(x))) \alpha (x) + i \mathrm{Re} ( H_1 (kA(x))) \beta (x) \Big ] e^{- ik B(x)} \, \diff x\Big ) e_k  \label{formE0}\\
E_1(a,b)[\alpha,\beta] &  
:=  \sum_{k\neq 0} \Big ( \int_\T K( kA(x)) \alpha(x) e^{- ik B(x)}\, \diff x\Big ) e_k \, . \label{formE1}
\end{align}
Note that $E_1(a,b)[\alpha,\beta]  $ is independent of $ \beta $.

The above decomposition  \eqref{SabE0E1}  induces 
the following decomposition of $ \tilde M_R^R $ in \eqref{ARRsmall}. 

\begin{lem}
{\bf (Decomposition lemma)}
The operator 
$ \tilde M_R^R $ in \eqref{ARRsmall} may be written as 
\be\label{ARRNlem}
 \tilde M_R^R  = 
 {\mathcal D} + {\mathcal N}_1 + {\mathcal N}_2 + {\mathcal N}_3 
\ee
where $ {\mathcal D} $ is the diagonal operator
\be\label{diagD} 
{\mathcal D}  := {\mathcal D}(a,b) := {\rm diag}(T^j_j (a,b))_{|j|> N}   
\ee
and 
\begin{align} 
{\mathcal N}_1 &  := {\mathcal N}_1 (a,b) := \big( T^j_k (a,b) \big)_{|j|,|k|> N, j \neq k} 
\label{defN1} \\
{\mathcal N}_2 &  := {\mathcal N}_2 (a,b) := -  M_R^L (M^L_L)^{-1} M^R_L \\
{\mathcal N}_3 & := {\mathcal N}_3 (a,b) := \Pi_R E_0\circ(\Pi_R E_1)^* + \Pi_R E_1\circ (\Pi_R E_0)^* + \Pi_R E_1\circ (\Pi_R E_1)^*  \, , \label{defN3}
\end{align}
with  matrix entries (cfr.~Lemma \ref{matrixentries})  
\begin{align}
			 T^j_k (a,b) 
			& =  2 \pi \int_\T \Big(
			\mathrm{Re} (H_1 ( kA(x))) \mathrm{Re} (H_1(jA(x))) + 
			\mathrm{Im} (H_1( kA(x))) \mathrm{Im} (H_1( jA(x)))   \Big) e^{i (j-k)B(x)}\, \diff x 
			\label{Tabjk1N} \\
			& =  \pi \int_\T r( kA(x)) \overline{r( jA(x))} e^{ i(j-k)(B(x)-A(x))}\, \diff x + \pi \int_\T \overline{r(kA(x))} r(jA(x)) e^{ i(j-k)(B(x)+A(x))}\, \diff x \, .  \label{Tabjk2N}   
			\end{align}
\end{lem}

\begin{proof}
According to the decomposition \eqref{SabE0E1} the linear operator $ M_R^R  $
in  \eqref{Adeco} is equal to 
\begin{align}
M_R^R  & = 
\Pi_R \diff S \circ \diff S^* \Pi_R \notag  = \Pi_R \diff S \circ (\Pi_R \diff S)^* \notag  \\
& = 
\underbrace{\Pi_R E_0 \circ (\Pi_R E_0)^*}_{= \, T(a,b)} \, + \, 
\underbrace{\Pi_R E_0 \circ (\Pi_R E_1)^* + 
\Pi_R E_1 \circ (\Pi_R E_0)^* + \Pi_R E_1\circ (\Pi_R E_1)^* }_{= \, {\mathcal N}_3(a,b)}
\, . \label{dSdS*}  
\end{align}
Then we write  in matrix form
\be\label{TplusDN}
T(a,b) [\gamma ] = \Pi_R E_0\circ (\Pi_R E_0)^* [\gamma] 
=  \sum_{|k|, |j| > N} 
T^j_k(a,b)  \widehat{\gamma}(j)  e_k  \equiv 
{\mathcal D} + {\mathcal N}_1  
\ee
where $ {\mathcal D} $ is the diagonal operator in \eqref{diagD} 
and  $ {\mathcal N}_1  $ is the off-diagonal operator defined in \eqref{defN1}.
Thus \eqref{ARRNlem}  is a direct consequence of   \eqref{ARRsmall}, \eqref{dSdS*} and \eqref{TplusDN}.  
Let us now compute the matrix entries $ T^j_k (a,b) $. 
By \eqref{formE0}, and since (cfr.~Lemma \ref{lem:adjo}) 
\[
E_0 (a,b)^* [\gamma] = - 2 \pi  
\Big ( \sum_{j \neq 0 } 
\mathrm{Im} (H_1( jA(x)) e^{ij B(x)} \widehat \gamma(j), i \sum_{j \neq 0 } 
\mathrm{Re}( H_1(jA(x))) e^{ i j B(x)} \widehat\gamma(j)\Big ) \, , 
\]
we deduce \eqref{Tabjk1N}. 
Finally  \eqref{Tabjk2N} follows 
recalling that $H_1(\theta) = e^{i\theta} r(\theta)$ and using 
the identity $\mathrm{Im} (z_1) \mathrm{Im}(z_2) + \mathrm{Re} (z_1) \mathrm{Re}(z_2) = \mathrm{Re} (z_1\bar z_2)=
\frac 12 (z_1\bar z_2 + \bar z_1z_2)$.
\end{proof}

We first prove the invertibility of the diagonal operator $ {\mathcal D} $ defined in \eqref{diagD}. 

\begin{lem}\label{lemD}
{\bf (Invertibility  of $ {\mathcal D}  $)} 
 Assume  \eqref{eq:alphabsmall}. 
 There is $ N_0 \in \N $ such that, for any $ N \geq N_0 $,
 the diagonal operator $ {\mathcal D}  $ in \eqref{diagD} is invertible and 
its matrix entries satisfy 
\be\label{Tjjd}
|T^j_j (a,b)| \sim {|j|}^{-1} \, . 
\ee
As a consequence there is $ C> 0 $ such that
\be\label{D-1s}
\| {\mathcal D} u \|_{s+1} \leq C \| u \|_s \, , \quad
\| {\mathcal D}^{-1} u \|_{s} \leq C \| u \|_{s+1} \, , \quad \forall  s \in \R \, .
\ee
\end{lem}
\begin{proof}
By \eqref{diagD}, \eqref{Tabjk1N} 
we have 
\[
{\mathcal D} = {\rm diag}(T^j_j (a,b))_{|j|> N} = 2 \pi\, 
\text{diag}_{|j| > N} \Big( \int_\T  |H_1(jA(x))|^2  \, \diff x \Big) = 
2 \pi \, 
\text{diag}_{|j| > N} \Big( \int_\T  |r(jA(x))|^2  \, \diff x \Big) 	\, ,
\]
since $ |H_1 (\theta)| = | r(\theta)| $ where $ r(\theta )$ is the function 
defined in \eqref{e:Jr}.
By \eqref{eq:alphabsmall} it results $ A (x) \geq A_* / 2 > 0 $ and 
therefore, by \eqref{e:r}-\eqref{firstHF},  for any $ |j| > N $ large enough,  
% $ |j A(x)| >  \frac12 | j | A_* > \theta_* $
 the estimate \eqref{Tjjd} follows.
\end{proof}

In view of  Lemma \ref{lemD}  we rewrite the linear operator  $ \tilde M_R^R $
in \eqref{ARRNlem} as 
\be\label{TabDabNab}
\tilde M_R^R = 
{\mathcal D} \big( \text{Id} + {\mathcal R}_1 + {\mathcal R}_2 + {\mathcal R}_3
\big)  
\ee
where
$  {\mathcal R}_i :=  {\mathcal D}^{-1}  {\mathcal N}_i $ for any $ i = 1,2,3 $. 

The invertibility of $ \tilde M_R^R $ and the tame estimates of its inverse 
will follow by Lemma \ref{lem:invertame}, 
once we show that each $  {\mathcal R}_i $, $ i = 1,2,3 $, 
is a $ 0$-tame operator satisfying the smallness 
assumption of Lemma \ref{lem:invertame}. 
\\[1mm]
{\bf Step 1: Estimates of ${\mathcal R}_1 $.} By 
the non-stationary-phase principle 
we  deduce the following lemma. 

\begin{lem}\label{lem:N1}
{\bf (Off-diagonal decay of ${\mathcal N}_1 $)}
Assume \eqref{eq:alphabsmall}. 
For any $ s \in \N $,  
 for any $ j \neq k $,  
\be\label{Njk}
%\lesssim_s \frac{1}{\sqrt{|k|}\sqrt{|j|}} \frac{1}{|j-k|^{s+1}} 
%\| (a, b) \|_{s+3}  
| [{\mathcal N}_1]^j_k| \leq C(s)  
\frac{\| (a, b) \|_{s+1}}{|k |^{\frac12} |j|^{\frac12} |j-k|^{s-1}} 
\, . 
\ee
\end{lem}

\begin{proof}
Let us 
prove that  the second term of 
$ [{\mathcal N}_1]^j_k = T_k^j (a,b)  $ in \eqref{Tabjk2N},  namely 
\[
\Sigma_k^j   := \pi \int_\T 
\overline{r( kA(x))} r(jA(x))
 e^{ i(j-k)(B(x)+ A(x))}\, \diff x   - \pi \int_\T 
\overline{r( kA_*)} r( jA_*)
 e^{ i(j-k)(x+ A_*)}\, \diff x \, , 
\]
(the last integral is zero since $ j \neq k $) satisfies a bound as \eqref{Njk}.
Recalling that $ A(x)  = A_* + a(x) $, 
$ B(x) = x + b(x) $ (cfr.~\eqref{AxBx}), and denoting
\[
A_\tau (x) := A_* + \tau a(x) \, , \quad B_\tau (x) := x + \tau b(x) \, , 
\quad \forall \tau \in [0,1] \, , 
\]  
we write 
\[
\begin{aligned}
\Sigma_k^j & = \int_0^1 \frac{d}{d\tau} \int_{\T} 
\overline{r(kA_\tau (x))} r( jA_\tau (x))
 e^{ i(j-k)(B_\tau (x)+ A_\tau(x))}\, \diff x  \diff \tau\\
 & = \int_0^1  \int_{\T} \sigma_\tau (x,j,k)
% \overline{r(2 \pi kA_\tau (x))} r(2\pi jA_\tau (x))
 e^{ i(j-k)(x + \tau b(x)  + \tau a(x) )}\, \diff x  \diff \tau
 \end{aligned}
\]
where
\begin{equation}\label{sigmakj}
\begin{aligned}
\sigma_\tau (x,j,k) 
& := \Big[  \overline{r' (k A_\tau (x))} k a(x)
r ( j A_\tau (x)) + 
\overline{r (k A_\tau (x))} r' (j A_\tau (x)) j a(x) \\
& \quad + \overline{r( kA_\tau (x))} r( jA_\tau (x)) 
 i(j-k)(b (x)+ a(x)) \Big]  e^{ i(j-k) A_*  } \, . 
\end{aligned}
\end{equation}
We apply the non-stationary phase Lemma \ref{lem:PS}
with
\[
u(x) = \sigma_\tau (x,j,k)  \, , \quad p(x) = \tau b (x) + \tau a(x)  \, , \quad  n =  j - k \, .
\]
Note that $ \| p \|_2 \leq \| b \|_2  + \| a \|_2 =: K $, $ \| p' \|_{C^0} \leq 
\| b \|_{C^0} + \| a \|_{C^0} < 1 / 2 $
by \eqref{eq:alphabsmall}, and for any $s\in \R$
\be\label{ps+1}
\| p \|_{s+1} \leq \| (a,b) \|_{s+1}  \, . 
\ee
By \eqref{sigmakj}, \eqref{eq:tameproduct}, Lemma \ref{lem:Sobosimbol}, \eqref{e:r},
\eqref{eq:alphabsmall}, and since $ 1 \leq |j-k|$, we deduce that,
for any $ s \geq s_0 > \tfrac12 $, 
\be\label{equs}
\| u \|_s \leq \sup_{\tau \in [0,1]} \| \sigma_\tau ( \cdot ,j,k)\|_s 
\lesssim_s \frac{|j-k|}{\sqrt{|j|} \sqrt{|k|}} \|(a,b) \|_{s} \, . 
\ee
Then, by \eqref{fasesta}, \eqref{ps+1} and \eqref{equs} 
we conclude that, for any $ j \neq k $, $ s \in \N $, 
\be\label{decays-1}
| \Sigma^j_k|  \lesssim_s 
% \frac{1}{\sqrt{|k|}\sqrt{|j|}}
\frac{1}{|j-k|^s} 
\Big( \| u \|_{s} + \| p \|_{s+1} \| u \|_1  \Big) 
\lesssim_s \frac{1}{\sqrt{|k|}\sqrt{|j|}} \frac{1}{|j-k|^{s-1}} 
\| (a, b) \|_{s+1}  
\ee
% Changing $ s$ with $ s + 2 $ in \eqref{decays-1}  we deduce   \eqref{Njk}. 
The first term in \eqref{Tabjk2N}  is  estimated  similarly. 
\end{proof}

Lemmata \ref{lem:N1} and 
\ref{lemD} imply off-diagonal decay of the matrix which represents the 
operator $ {\mathcal R}_1 
= {\mathcal D}^{-1} {\mathcal N}_1 $.  

\begin{lem}\label{decayR}
For any $ s \in \N $,  for any $ |j|, |k| > N $,  % \in \Z \setminus \{0\} $, 
\be\label{decayR1off}
|[{\mathcal R}_1]^j_k| \leq C(s)  \frac{ \| (a, b)\|_{s+1}}{ |j-k|^{s-\frac32}} \, , \quad 
\forall  j \neq k \, .
\ee
\end{lem}

\begin{proof}
By \eqref{Tjjd} and \eqref{Njk}  
we have, for any $ s \in \N $,  
\be\label{Rjk}
|[{\mathcal R}_1]^j_k| = 
|[{\mathcal D}^{-1}]^k_k| |[{\mathcal N}_1]^j_k| 
 \lesssim_s  \frac{|k|^{\frac12} \| (a, b) 
\|_{s+1}}{|j |^{\frac12} |j-k|^{s-1}} \, , \quad \forall j \neq k \, .
\ee
We now decompose the indices $ (j,k)$ according if
\[
(i) \quad |j-k| < \frac{|k|}{2} \, , \qquad (ii) \quad |j-k| \geq \frac{ |k|}{2} \, . 
\]
In case ($i$) we have $ ||j|-|k|| \leq |j-k| < \frac{ |k|}{2} $ and therefore 
$ |j| \sim | k | $ and, by \eqref{Rjk}, we deduce the bound
\be\label{caseR1}
|[{\mathcal R}_1]^j_k| \lesssim_s  \frac{ \| (a, b) \|_{s+1}}{ |j-k|^{s-1}} \, , \quad 
\forall s \in \N \,  .
\ee
In case ($ii$) we have,  since $ |j| > N \geq 1  $, for any 
$ s \in \N $, 
\be\label{caseR2}
|[{\mathcal R}_1]^j_k| \lesssim_s  \frac{ |k|^{\frac12} \| (a, b) \|_{s+1}}{ |j-k|^\frac12  |j-k|^{s-\frac32}}
\lesssim_s  \frac{ \| (a, b) \|_{s+1}}{   |j-k|^{s-\frac32}}  \, .
\ee
The bounds  \eqref{caseR1}-\eqref{caseR2} directly imply \eqref{decayR1off}. 
\end{proof}

Lemma \ref{decayR} implies that 
$ {\mathcal R}_1 $ % is a bounded operator
has  finite 
$ s$-decay norm defined in \eqref{sdecay}.
 
\begin{lem} \label{decayR1}
{\bf (Off-diagonal decay of $ {\mathcal R}_1 $)}
For any $ s \geq 0  $ we have 
\[ 
| {\mathcal R}_1 |_s \leq C(s)   \| (a, b) \|_{s+4}  \, . 
\]
As a consequence  ${\mathcal R}_1 $ is a $ 0$-tame operator with a tame constant 
$ C_{{\mathcal R}_1} (s) \leq C(s) \| (a, b) \|_{s+4}   $, cfr.~\eqref{tameMsa}.
\end{lem}

{\bf Step 2: Estimates of 
$ {\mathcal R}_2  $. }
The  operator $ {\mathcal R}_2 
=   - {\mathcal D}^{-1}  M_R^L (M^L_L)^{-1} M^R_L $
is a small finite rank operator, since  $M_R^L $ and $ M^R_L $
are small finite rank operators  as in \eqref{Aalpha0}. The following estimates hold.

\begin{lem}\label{lem:restireso}
{\bf (Bound of $ {\mathcal R}_2 $)} Assume \eqref{eq:alphabsmall}. 
The operator $ {\mathcal R}_2  $ 
satisfies for any  $ s \geq \frac 32 $, any $ \gamma \in H^0_0(\T) $,   
\[
\| {\mathcal R}_2 \gamma \|_s \leq C(s,N) \| (a,b) \|_{s+ \frac52} \| \gamma \|_0 \, . 
\]
\end{lem}

\begin{proof}
By \eqref{D-1s} and  Lemmata \ref{lem:coupling} and  \ref{invALL}. %  and \ref{lemD}.   
\end{proof}

{\bf Step 3: Estimate of $ {\mathcal R}_3 $. }
Exploiting the fact that $ E_1 $ in \eqref{formE1} 
is bounded as an operator of order $ - \frac32 $
we now prove that the operator $ {\mathcal N}_3 $ 
in \eqref{defN3} is  small of size $ O(1/ \sqrt{N}) $ 
as a bounded operator of order 
$ - 1  $. 
Let us start showing the following.

\begin{lem}\label{restoE1}
For any $ s \geq 3 $ we have 
\begin{equation}
\label{PRE1picco}
\| \Pi_R E_1  [\alpha,\beta] \|_s \leq 
C(s) \frac{1}{\sqrt{N}}  \| \alpha  \|_{s- 1}   + C(s) \frac{1}{\sqrt{N}}  \| (a,b) \|_{s}  \| \alpha \|_{2}   \, . 
\end{equation}
For any $ s \geq 2  $ we have 
\begin{equation}
 \label{PRE1piccoad}
\| (\Pi_R E_1)^*   \gamma \|_s \leq 
C(s) \frac{1}{\sqrt{N}}  \| \gamma \|_{s- 1}  +
C(s) \frac{1}{\sqrt{N}}  \| (a,b) \|_{s}  \| \gamma \|_{1}  \, . 
\end{equation}
\end{lem}

\begin{proof}
The operator $ E_1 $ is given in \eqref{formE1} with $ K $ defined in \eqref{defK}, and thus
\begin{align}
\Pi_R E_1(a,b)[\alpha,\beta] 
& := 
 \sum_{|k| > N } \Big ( \int_\T K(kA(x)) \alpha(x) e^{- ikB(x)}\, \diff x\Big ) e_k \notag  \\
& = 
 \sum_{k \in \Z } \Big (\int_\T \mathrm{Re} \big(  r'( kA(x)) e^{i kA(x)} \big) \chi_{> N} (k) \alpha(x) e^{- ikB(x)}\, \diff x\Big ) e_k \label{PIRE1}
\end{align}
where $ \chi_{> N} (z)$ is a $ C^\infty $ cut-off function 
equal to $ 1 $ for $ |z| \geq N +1  $ and $ 0 $ for $ |z| \leq N  $. 
 The operator in \eqref{PIRE1} is the sum 
\be\label{pezzE1}
\begin{aligned}
& \ \ \frac12 \sum_{k \in \Z } 
 \Big (\int_\T  r'(kA(x)) e^{i k A_* } \chi_{>N} (k) 
 \alpha(x) e^{-ik(x+b(x) -a(x))}\, \diff x\Big ) e_k \\
 & +
 \frac12 \sum_{k \in \Z } 
 \Big (\int_\T  \overline{r'(kA(x))} e^{-i k A_* }  \chi_{>N} (k) 
 \alpha(x) e^{-ik(x+ b(x) + a(x))}\, \diff x\Big ) e_k \, , 
 \end{aligned}
 \ee
 namely of operators of the form \eqref{eq:E*App} 
 with  
$ u \equiv \alpha $, 
symbols $ h(x,k) \equiv \tfrac12 \, \overline{r'(kA(x))} e^{- i k A_*}\chi_{>N} (k) $ 
and its conjugate, and $ p(x) =  b(x) \pm a(x)  $.
By  Lemmata \ref{lem:J1} 
 and \ref{lem:Sobosimbol} 
\be\label{h32b}
 |h|_{- \frac32,s} \lesssim_s 1 + \| a \|_s \, , \quad \forall s\geq 0 \, . 
 \ee 
 The function $ p $ satisfies, for any $ s \in \R $,  
\be\label{psb}
\| p \|_s  \leq  \| b\|_s  + \| a\|_s  \lesssim_{s} \| (a,b) \|_s \, . 
 \ee
Regarding $ h $
as a  symbol of order $ - 1 $,
 by Lemma \ref{lem:piccolosimbolo} and \eqref{h32b} we deduce that 
\be\label{simbolo-12}
| h |_{-1,s} \leq \frac{1}{\sqrt{N}} | h |_{-\frac32,s} 
\lesssim_s \frac{1}{\sqrt{N}} (1+ \| a \|_{s}), \quad \forall s\geq 0  \, . 
\ee
Thus \eqref{PRE1picco} follows by Lemma \ref{lem:Estargene} for any $ s
\geq \max\{1, s_0 +1 - \sigma \} = 3 $
and \eqref{simbolo-12}, \eqref{psb}.  

Finally, since the adjoint $ (\Pi_R E_1)^* $
is  the sum of adjoint of operators of the form \eqref{pezzE1}, 
the tame estimates \eqref{PRE1piccoad}  follow, for any $ s
\geq s_0 +1 = 2 $, 
by Lemma \ref{lem:Eusmall}  for a symbol of order $ - 1$ and \eqref{simbolo-12}, \eqref{psb}.   
\end{proof}

The operators $ \Pi_R E_0 $ and 
$ (\Pi_R E_0)^* $, where $E_0$ is defined as in \eqref{formE0},   have  order $ - \frac12 $ 
and satisfy tame estimates 
as $ \diff S $ and $ \diff S^* $, cfr.~Lemmata \ref{lem:tamedS} 
and \ref{dSadj}. 

\begin{lem}\label{restoE2}
For any $ s \geq 5/ 2  $ we have 
\be \label{PRE0}
 \| \Pi_R E_0  [\alpha,\beta]  \|_s \leq 
 C(s) \| (\alpha,\beta)  \|_{s- \frac12}  + C(s)  \| (a,b) \|_{s+ \frac12}  \| (\alpha,\beta)  \|_{2} .
\ee
For any $ s \geq 2 $
\be \label{PRE0ad}
\| (\Pi_R E_0)^*   \gamma \|_s \leq 
  C(s) \| \gamma \|_{s- \frac12}  + C(s) \| (a,b) \|_{s}  \| \gamma \|_{\frac32}  \, . 
\ee
\end{lem}

We deduce that the operator
${\mathcal R}_3 $
is bounded, of size $ 1/ \sqrt{N} $, and satisfies 
the tame estimates  \eqref{Rhs-1} below, where we note that 
the constant in front of the high norm $  \| \gamma \|_{s} $ is {\it independent} of $ s $. 
\begin{lem}
\label{lemR2}
{\bf (Estimates of $ {\mathcal R}_3 $)}
The operator
%\be\label{R2tame}
$ {\mathcal R}_3 =
{\mathcal D}^{-1}  {\mathcal N}_3 $ 
 satisfies, for any $ s \geq 2 $, 
\be\label{Rhs-1}
\| {\mathcal R}_3 \gamma \|_s \leq \frac{1}{\sqrt{N}}  \| \gamma \|_{s}  + 
C_3 (s) \frac{1}{\sqrt{N}} \big( 1 + \| (a,b) \|_{s+ \frac32}\big) 
 \| \gamma \|_{\frac32}   \, . 
\ee
As a consequence  
the operator ${\mathcal R}_3 $ 
is $ 0$-tame with  tame constant $ C_3 (s) \frac{1}{\sqrt{N}} \big( 1 + \| (a,b) \|_{s+ \frac32}\big) $.  
\end{lem}

\begin{proof}
By \eqref{defN3}, \eqref{D-1s} and Lemmata  \ref{restoE1} and \ref{restoE2} we deduce that, 
for any $ s \geq 2  $, 
\begin{align} 
\| {\mathcal R}_3 \gamma \|_s 
& \leq 
C' (s) \frac{1}{\sqrt{N}} \| (a,b) \|_{s+\frac32}  \| \gamma \|_{\frac32}   + 
C' (s) \frac{1}{\sqrt{N}}  \| \gamma \|_{s- \frac12} \notag  \\
& \leq 
C' (s) \frac{1}{\sqrt{N}} \| (a,b) \|_{s+\frac32}  \| \gamma \|_{\frac32}   + 
C' (s) \frac{1}{\sqrt{N}} \Big(  \epsilon \| \gamma \|_{s} + 
C(\epsilon)\| \gamma \|_{\frac32} \Big)    \notag
\end{align}
by the interpolation estimate \eqref{asintpo}
with $ v = 1 $, $ \gamma \equiv u $, $ s_1 = \tfrac32  $, $ q_1 = s- 2 $, 
$ q_2 = \frac12 $. By choosing 
$ \epsilon = 1 / C' (s) $ we deduce \eqref{Rhs-1}.
\end{proof}

\noindent
{\bf Conclusion.} 
By Lemmata \ref{decayR1}, \ref{lem:restireso} and \ref{lemR2} 
the operator $ {\mathcal R}_1 + {\mathcal R}_2 + {\mathcal R}_3 $ is a 
$ 0$-tame operator with, for any $ s \geq s_1 \equiv 2 $,   a tame constant (cfr.~\eqref{sigmatame})
\be\label{tameR1R2R3}
C_{{\mathcal R}_1 + {\mathcal R}_2 + {\mathcal R}_3} (s) \leq 
\frac{C_4(s)}{\sqrt{N}}(1+ \| (a,b) \|_{s+ \frac32}) + C(s,N) \| (a,b) \|_{s+ 4} \, .
\ee
For $ s = s_1 $ we have 
\[ 
C_{{\mathcal R}_1 + {\mathcal R}_2 + {\mathcal R}_3} (s_1) < 
\frac{2C_4(s_1)}{\sqrt{N}} + C(s_1,N) \| (a,b) \|_{s_1+ 4} \leq \frac14 % 1 / 4  
\]
by taking
\be\label{abovesmm}
N := \max \Big\{ \big[ \big(16 C_4 (s_1) \big)^2 \big] + 1, N_0 \Big\}   \qquad \text{and} \qquad 
 C(s_1,N) \| (a,b) \|_{s_1+ 4} \leq \frac18 
\ee
 where $N_0$ is given by Lemma \ref{lemD}.  The inequality \eqref{abovesmm}
is  implied by the smallness condition \eqref{eq:alphabsmall}
with  $ \delta  := \frac{1}{8 C(s_1,N)} $.  
%  \eqref{Rs0}-\eqref{Rstame}. 
As a consequence 
Lemma  \ref{lem:invertame} implies that 
$ \text{Id} + {\mathcal R}_1 + {\mathcal R}_2 + {\mathcal R}_3$ is invertible 
and, recalling  \eqref{TabDabNab},  
\[
(\tilde M_R^R)^{-1} = 
 \Big( \text{Id} + {\mathcal R}_1 + {\mathcal R}_2 + {\mathcal R}_3
\Big)^{-1} {\mathcal D}^{-1} \, . 
\]
Finally using also  \eqref{1+Rstame}, \eqref{tameR1R2R3} 
and \eqref{D-1s}, %  since $N\geq N_0$,  
we obtain, for any $ s \geq s_1 = 2  $ ($ N $ is fixed in \eqref{abovesmm} and $ s_1 = 2 $)
\[
\begin{aligned}
\| (\tilde M_R^R)^{-1} \gamma \|_s 
& =  \Big\| \Big( \text{Id} + {\mathcal R}_1 + {\mathcal R}_2 + {\mathcal R}_3 \Big)^{-1} {\mathcal D}^{-1} \gamma \Big\|_s \\
& \leq 2  \| {\mathcal D}^{-1} \gamma  \|_{s} +  
C_5 (s) \big( 1+  \| (a,b) \|_{s+ 4}  \big)   \| {\mathcal D}^{-1} 
\gamma  \|_{s_1} \\
& \leq \max \{ 2, C_5 (s)\}  \| \gamma \|_{s+1} +  C_5(s) \| (a, b) \|_{s+4} \|  \gamma  \|_{3} \, .   
\end{aligned} 
\]
This proves Proposition \ref{invpropAtilde}, thus completing the proof of Theorem \ref{thm:ri}.

\section{Proof of Theorem \ref{thm:main}}\label{sec:NM}

Theorem \ref{thm:main} is proved by the application of a Nash--Moser 
implicit function theorem.  
We report here a simplified version of \cite[Theorem 5.1 and Corollary 5.2]{Ambrozio:2021}   
 (in our application we have 
a right inverse of the linearized operator, not just an approximate one).
We  refer to \cite{Hamilton}, cfr.~\eqref{defSmtame} below, 
 for the notion of a smooth tame map.  
% according to  Definition 2.1.1 of

\begin{thm}[\cite{Ambrozio:2021}]\label{thm:CM}
Let $ {\mathcal F}  $ and $ {\mathcal H } $ 
be tame Fr\'echet spaces and let 
$ {\mathcal S} : {\mathcal U} \subset {\mathcal F}  \mapsto  {\mathcal H}   $
be a smooth tame map defined on an open set containing the origin
${\mathcal U} \subset {\mathcal F}  $, 
with $ {\mathcal S} (0) = 0 $. 
 Suppose  that, for any $ u \in {\mathcal U} $,  there exists 
 a  right inverse $ {\mathcal R} (u)$ of $ \diff {\mathcal S} (u) $ 
 with the property that 
\[ {\mathcal R} : {\mathcal U} \times {\mathcal H}  \mapsto {\mathcal F} \, ,
\quad (u,\gamma) \mapsto  {\mathcal R } ( u ) [\gamma ] \,,
\]   
is a  smooth tame map. 
Then, % there exists a neighborhood $ {\mathcal W} \subset {\mathcal U} $ with
%$ 0 \in {\mathcal W} $ and a smooth tame map 
%$$
%\Gamma : \ker \diff {\mathcal S} (0) \cap  {\mathcal W} \to  {\mathcal S}^{-1} (0)
%$$  
%such that
%$ \Gamma (0) = 0 $ and $ d \Gamma (0) [v]= v $ for any  
%$ v \in  \ker \diff {\mathcal S} (0) $. 
 for any $  v \in  \ker \diff {\mathcal S} (0) $ there exists a smooth 
$ 1 $-parameter family $ u (\tau)  $, $ \tau \in (-\delta, \delta)$, such that
\[
u(0) = 0 \, , \quad \frac{d}{d\tau}\Big |_{\tau=0} u(\tau) = v \, , \qquad 
{\mathcal S} (u (\tau)) = 0  \qquad \text{for any} \quad \tau \in (-\delta, \delta) \, . 
\] 
\end{thm}

We apply Theorem \ref{thm:CM}
 to the nonlinear map $ {\mathcal S}  := S $ in \eqref{eq:S} 
defined on the Fr\'echet spaces 
$ {\mathcal F} := C^\infty (\T) \times  C^\infty (\T) $ and 
$ {\mathcal H} := C^\infty_0 (\T) $
and open set
$ {\mathcal U} := \{ u = (a,b) \, : \, \| u \|_6 < \delta  \} $. By Lemma \ref{lem:tameS}
we know that $S  $ is a tame nonlinear map, in the sense of Definition 2.1.1 of \cite{Hamilton}. 

\begin{lem}  \label{adjsmtame}
The map $S :  {\mathcal U}  \subset {\mathcal F} \to {\mathcal H}   $ is smooth tame. 
The map  $ \diff S^*  : \{ {\mathcal U}  \subset {\mathcal F}\} \times {\mathcal H}  
\to {\mathcal F}    $ 
is   smooth tame. 
 \end{lem}

\begin{proof}
Let us  prove that for any $ m \in \N $ the derivative $ \diff^m S $
is tame, i.e. there is $ \sigma (m) $ such that, for any $ s \geq 5/2 $, 
\be\label{defSmtame}
\begin{aligned}
 \big\| \diff^m S(a,b)&[(\alpha_1,\beta_1), \ldots, (\alpha_m,\beta_m)] \big\|_{s} 
 \\
& \lesssim_s
\| (a, b)\|_{s + \sigma(m)} 
\prod_{i=1}^m 
\| (\alpha_i, \beta_i)\|_{s_0 + \sigma(m)} + 
\sum_{j=1}^m \| (\alpha_j, \beta_j)\|_{s + \sigma(m)} \prod_{i \neq j}
\| (\alpha_i, \beta_i)\|_{s_0 + \sigma(m)} \, . 
\end{aligned}
\ee
For $ m = 1, 2 $ the bounds \eqref{defSmtame} 
follow by Lemmata \ref{lem:tamedS} and \ref{lem:tamedS2}. 
For $ I \subset \{ 1, \ldots, m \}  $ we denote by $ |I| $ its cardinality,
 $ |\varnothing | = 0 $ 
%if $ I = \emptyset $
 and $ J_1^{(0)} \equiv J_1  $. 
For any $ m \in \N $ it results (by differentiating \eqref{diffSab}) % \eqref{eq:coefficientsdS^2})
\begin{align}
&  \diff^m S(a,b)[(\alpha_1,\beta_1), \ldots, (\alpha_m,\beta_m)] 
 \label{eq:coefficientsdS^m}  \\
 & =    \sum_{k \neq 0}
 \sum_{I\subset\{1,\ldots,m\}} 
(-i)^{|I^c|}   k^{m-1} \Big(  \int_\T    J_1^{(|I|)} ( k A(x))  \Pi_{i_1 \in I}\alpha_{i_1} (x) 
\Pi_{i_2 \in I^c} \beta_{i_2} (x)   e^{- i k B(x)}  \, \diff x \Big) e_k \, . \nonumber 
\end{align}
From \eqref{eq:coefficientsdS^m} we have that 
$\diff^m S(a,b)$ is sum of operators of the form $E^*$ in \eqref{eq:E*App} with 
$u =  \Pi_{i_1 \in I}\alpha_{i_1} 
\Pi_{i_2 \in I^c} \beta_{i_2 } $ 
and symbols $h$ as in \eqref{eq:symbolh} of order $\sigma= m - \frac32 $ 
 by Lemmata \ref{lem:J1} and \ref{lem:Sobosimbol} satisfying
 $ | h |_{m - \frac32,s} \lesssim_{s,m} 1 + \| a \|_s $ for any $ s \geq 0 $. Therefore, Lemma \ref{lem:Estargene} with $ p(x) = a(x) \pm b(x) $ implies that  $\diff^m S $ satisfies \eqref{defSmtame} with $\sigma (m) = m - \frac12 $. 
  Similarly, arguing as in Lemma \ref{dSadj}, by  differentiating \eqref{eq:dS*},  
  we deduce that $ \diff S^*   $ is  smooth tame. 
\end{proof}

We know from Theorem~\ref{thm:ri} that, for any $(a,b)\in \mathcal U$, $\diff S(a,b)$ has a right inverse $R(a,b)$ as in \eqref{Rab} which satisfies tame estimates \eqref{Rabfinal}.
In order to prove that $ R(a,b)$  is smooth tame 
we use  Lemma \ref{adjsmtame} and that
$ (\diff S (a,b)\circ \diff S(a,b)^*)^{-1}  $ is smooth tame. This in turn follows because 
$ \diff S (a,b)\circ \diff S(a,b)^* $ is smooth tame, 
being composition of smooth tame maps, see \cite[Part II, Thm 2.1.6]{Hamilton}.
Moreover $ \diff S (a,b)\circ \diff S(a,b)^* $   is invertible and its inverse is tame by 
Proposition \ref{prop:key}. Finally
 Theorem \cite[Part II.3, Thm 3.1.1]{Hamilton} implies $ (\diff S (a,b)\circ \diff S(a,b)^*)^{-1}  $ is smooth tame.  

All the assumptions of Theorem \ref{thm:CM} are thus satisfied by $ S(a,b) $.
Since Lemma \ref{lem:dStrivial}  shows that 
$ \ker \diff {\mathcal S} (0) $ is equal to the subspace $ V_{A_*} \subset {\mathcal F } $ 
 defined in 
\eqref{SDABstar}, %  Thus the application of Theorem \ref{thm:CM} 
%\subset  {\mathcal F}$  we deduce 
it implies  Theorem \ref{thm:main}.

\bibliographystyle{plain}
\bibliography{_biblio}

\end{document}